\documentclass[11pt]{amsart}

\DeclareMathOperator{\Lie}{Lie}
\DeclareMathOperator{\Rad}{Rad}
\DeclareMathOperator{\ad}{ad}
\usepackage{amsmath,amssymb,amsthm,graphics,amscd}
\usepackage{longtable}
\usepackage{cite,xypic}
\usepackage{color}
\usepackage{graphicx}
\usepackage{epsf}
\usepackage{fancyhdr,hyperref}
\usepackage{pdflscape}
\hypersetup{backref,
colorlinks=true,backref}

  \renewenvironment{thebibliography}[1]{
    \begin{oldthebibliography}{#1}
      \setlength{\parskip}{0ex}
      \setlength{\itemsep}{0ex}
  }
  {
    \end{oldthebibliography}
  }
  
\setlongtables
\begin{document}

\newcounter{rownum}
\setcounter{rownum}{0}
\newcommand{\ab}{\addtocounter{rownum}{1}\arabic{rownum}}
\newcommand{\im}{\mathrm{im}}
\newcommand{\rk}{\mathrm{rk}}
\newcommand{\x}{$\times$}
\newcommand{\bb}{\mathbf}

\newcommand{\Ind}{\mathrm{Ind}}
\newcommand{\Char}{\mathrm{char}}
\newcommand{\hra}{\hookrightarrow}
\newtheorem{lemma}{Lemma}[section]
\newtheorem{theorem}[lemma]{Theorem}
\newtheorem*{TA}{Theorem A}
\newtheorem*{TB}{Theorem B}
\newtheorem*{TC}{Theorem C}
\newtheorem*{CorC}{Corollary C}
\newtheorem*{TD}{Theorem D}
\newtheorem*{TE}{Theorem E}
\newtheorem*{PF}{Proposition E}
\newtheorem*{C3}{Corollary 3}
\newtheorem*{T4}{Theorem 4}
\newtheorem*{C5}{Corollary 5}
\newtheorem*{C6}{Corollary 6}
\newtheorem*{C7}{Corollary 7}
\newtheorem*{C8}{Corollary 8}
\newtheorem*{claim}{Claim}
\newtheorem{cor}[lemma]{Corollary}
\newtheorem{conjecture}[lemma]{Conjecture}
\newtheorem{prop}[lemma]{Proposition}
\newtheorem{question}[lemma]{Question}
\theoremstyle{definition}
\newtheorem{example}[lemma]{Example}
\newtheorem{examples}[lemma]{Examples}
\theoremstyle{remark}
\newtheorem{remark}[lemma]{Remark}
\newtheorem{remarks}[lemma]{Remarks}
\newtheorem{obs}[lemma]{Observation}
\theoremstyle{definition}
\newtheorem{defn}[lemma]{Definition}

  \def\hal{\unskip\nobreak\hfil\penalty50\hskip10pt\hbox{}\nobreak
  \hfill\vrule height 5pt width 6pt depth 1pt\par\vskip 2mm}

\renewcommand{\labelenumi}{(\roman{enumi})}
\newcommand{\Hom}{\mathrm{Hom}}
\newcommand{\Int}{\mathrm{int}}
\newcommand{\Ext}{\mathrm{Ext}}
\newcommand{\opH}{\mathrm{H}}
\newcommand{\soc}{\mathrm{Soc}}
\newcommand{\SO}{\mathrm{SO}}
\newcommand{\la}{\langle}
\newcommand{\ra}{\rangle}
\newcommand{\Sp}{\mathrm{Sp}}
\newcommand{\SL}{\mathrm{SL}}
\newcommand{\GL}{\mathrm{GL}}
\newcommand{\diag}{\mathrm{diag}}
\newcommand{\End}{\mathrm{End}}
\newcommand{\tr}{\mathrm{tr}}
\newcommand{\Stab}{\mathrm{Stab}}
\newcommand{\red}{\mathrm{red}}
\newcommand{\Aut}{\mathrm{Aut}}
\renewcommand{\H}{\mathcal{H}}
\renewcommand{\u}{\mathfrak{u}}
\newcommand{\Ad}{\mathrm{Ad}}
\newcommand{\N}{\mathcal{N}}
\newcommand{\NN}{\mathbb{N}}
\newcommand{\Z}{\mathbb{Z}}
\newcommand{\gl}{\mathfrak{gl}}
\newcommand{\g}{\mathfrak{g}}
\newcommand{\F}{\mathbb{F}}
\newcommand{\m}{\mathfrak{m}}
\renewcommand{\b}{\mathfrak{b}}
\newcommand{\p}{\mathfrak{p}}
\newcommand{\q}{\mathfrak{q}}
\renewcommand{\l}{\mathfrak{l}}
\newcommand{\h}{\mathfrak{h}}
\renewcommand{\t}{\mathfrak{t}}
\renewcommand{\k}{\mathfrak{k}}
\renewcommand{\c}{\mathfrak{c}}
\renewcommand{\r}{\mathfrak{r}}
\newcommand{\n}{\mathfrak{n}}
\newcommand{\s}{\mathfrak{s}}
\newcommand{\z}{\mathfrak{z}}
\newcommand{\pso}{\mathfrak{pso}}
\newcommand{\so}{\mathfrak{so}}
\renewcommand{\sl}{\mathfrak{sl}}
\newcommand{\psl}{\mathfrak{psl}}
\renewcommand{\sp}{\mathfrak{sp}}
\newcommand{\Ga}{\mathbb{G}_a}
\newcommand{\Gm}{\mathbb{G}_m}

\newenvironment{changemargin}[1]{%
  \begin{list}{}{%
    \setlength{\topsep}{0pt}%
    \setlength{\topmargin}{#1}%
    \setlength{\listparindent}{\parindent}%
    \setlength{\itemindent}{\parindent}%
    \setlength{\parsep}{\parskip}%
  }%
  \item[]}{\end{list}}

\parindent=0pt
\addtolength{\parskip}{0.5\baselineskip}

\subjclass[2010]{17B45}
\title{On the smoothness of normalisers, the subalgebra structure of modular Lie algebras and the cohomology of small representations}
\author{Sebastian Herpel}
\address{Ruhr Universit\"at Bochum \\ Bochum, Germany} 
\email{sebastian.herpel@rub.de {\text{\rm(Herpel)}}}

\author{David I. Stewart}
\address{University of Cambridge\\ Cambridge, UK} \email{dis20@cam.ac.uk {\text{\rm(Stewart)}}}
\pagestyle{plain}
\begin{abstract}
We provide results on the smoothness of normalisers in connected reductive algebraic groups $G$ over fields $k$ of positive characteristic $p$. Specifically we we give bounds on $p$ which guarantee that normalisers of subalgebras of $\g$ in $G$ are smooth, i.e.\ so that the Lie algebras of these normalisers 
coincide with the infinitesimal normalisers.

%We apply these results to investigate aspects of the Lie subalgebra structure of the Lie algebra $\g$, obtain information about maximal and maximal solvable subalgebras of $\g$. 
One of our main tools is to exploit cohomology vanishing of small dimensional
modules. Along the way, we obtain complete reducibility results for small
dimensional modules in the spirit of similar results due to Jantzen, Guralnick, Serre and Bendel--Nakano--Pillen.
\end{abstract}
\maketitle
{\small \tableofcontents}
\section{Introduction}
Let $G$ be an affine group scheme over an algebraically closed field $k$. 
We say $G$ is smooth if $\dim\Lie(G)=\dim G$.
A famous theorem of Cartier states that every affine group over a field of characteristic zero is smooth. 
Therefore, in this situation, the category of smooth group schemes is closed under the 
scheme-theoretic constructions of taking centres, centralisers, normalisers and transporters.
However, Cartier's theorem fails rather comprehensively in positive characteristic.
A classic example of a non-smooth algebraic group is the group scheme $\mu_p$ whose points are the $p$th roots of unity;
this is not smooth over a field of characteristic $p$---its Lie algebra is $1$-dimensional, but its $k$-points consist just of the identity element.
Furthermore, since $\mu_p$ is also the scheme-theoretic centre of $\SL_p$,
the centre of this reductive\footnote{We call a smooth algebraic group $G$ reductive provided that $R_u(G^0)=1$.} group is also not smooth over a field of characteristic $p$.
This means that the group-theoretic centre of $\SL_p$ misses important
infinitesimal information about the centre (for instance, the fact that $\SL_p$ is
not adjoint).

Nonetheless, centralisers are usually smooth. For example, it is a critical result of Richardson \cite[Lem.~6.6]{Ric67}, used extensively in the theory of nilpotent orbits, that the centraliser $G_e=C_G(e)$ of an element $e$ of $\g=\Lie(G)$ is smooth whenever $p$ is a very good prime for $G$.\footnote{Recall that $p$ is good if the following holds: $p$ is not $2$ if $G$ contains a factor not of type $A$, $p$ is not $3$ if $G$ contains an exceptional factor and $p$ is not $5$ if $G$ contains a factor of type $E_8$. The prime $p$ is very good if it is good and it does not divide $n+1$ for any factor of $G$ of type $A_n$.}  (Note that smoothness of the centraliser, or what is the same, the separability of the orbit map $G\to G\cdot e$ can be restated as $\Lie(G_e(k))=\c_\g(e)$.) In fact the centralisers of subgroup schemes of a connected reductive group $G$ are usually smooth:
work of Bate--Martin--R\"ohrle--Tange and the first author (cf.\ Proposition \ref{prop:pretty})
gives precise information on the characteristic $p$ of $k$,
depending on the root datum of $G$, for centralisers of all subgroup schemes of $G$ to be smooth.
It suffices, for instance, for $p$ to be very good for $G$.
Furthermore, centralisers of all subgroup schemes of $\GL_n$ are smooth.

The situation for normalisers is much less straightforward, which may explain why results in this direction have been unforthcoming until now.
For example, even when $G=\GL_n$, for any $n\geq 3$ and any $p>0$ an arbitrary prime, there are connected smooth subgroups of $G$ with non-smooth normalisers (see Lemma \ref{example:smoothUnipotent} below).
In light of this situation, perhaps it is surprising that there are any general situations in which normalisers of subgroup schemes are smooth. However, we prove that for sufficiently large $p$ depending on 
the connected reductive algebraic group $G$,
(a) all normalisers of height one subgroup schemes (in fact the normalisers of all subspaces of the Lie algebra of $G$); and (b) all normalisers of connected reductive subgroups are indeed smooth. Theorem \ref{thmreductivenorms} makes (b) precise and the proof is a straightforward reduction to the case of centralisers. Our main result follows.

\begin{TA}
There exists a constant 
$c=c(r)$ such that if $p>c$ and $G$ is any connected reductive group of rank $r$ then all normalisers $N_G(\h)$ of all subspaces $\h$ of $\g$ are smooth.

More precisely, let $d$ be the dimension of a minimal faithful
representation of $G$.
Then all normalisers of subspaces of $\g$ are smooth provided that
$p > 2^{2d}$.
In particular, if $G=\GL_n$ we may take $p > 2^{2n}$.
\end{TA}

\begin{remarks} (a). Clearly, the constant $c(r)$ in the theorem may be defined as $2^{2d'}$ for
$d'$ the maximal dimension of a minimal faithful module 
of a connected reductive group of rank $r$.

(b). Note that the maximum is finite since there are only a finite number of isomorphism types of connected reductive groups of a given rank over an algebraically closed field $k$. Each of these arises by base change from a split reductive group defined over the integers, so one can consider the theorem as a statement that for a fixed group $G_\Z$, the conclusion holds for each reduction modulo $p$ of $G_\Z$, whenever $p$ is sufficiently large.
\end{remarks}

It is natural to ask if lower bounds for the constant $c$ in Theorem A exist. In \S\ref{sec:examples}, we present a menagerie of examples where smoothness of normalisers fails; in particular, in Example \ref{fibonaccimotherfucker} we give a $p$-subalgebra of $\gl_{2n+12}$ with non-smooth normaliser whenever $p|F_n$, the $n$th Fibonacci number. Since $F_n\sim1.6^n$ and infinitely many Fibonacci numbers are expected to be prime, we conclude that $c(G)$ should grow exponentially with the rank of $G$. 
In other words the bound on $p$ in the theorem is likely to be of the right order.

The obstruction to finding linear bounds for $c$ comes from the fact that one cannot, in general, lift the maximal tori of Lie-theoretic normalisers to group-theoretic normalisers. However, many interesting subalgebras of $\g$ have normalisers which are generated by nilpotent elements (such as maximal semisimple subalgebras). Adding in this extra, natural hypothesis gives rise to much better bounds. In the following theorem let $h=h(G)$ denote the Coxeter number of (the root system $\Phi$ of) $G$. If $\Phi$ is reducible, then $h$ is taken as the maximum over all components.

\begin{TB} 
\begin{enumerate}\item 
Let $G$ be a reductive algebraic group and let $d$ be as in Theorem A.
Suppose $p>d+1$. 
Then all normalisers $N_G(\h)$ of $p$-subalgebras $\h$ are smooth
whenever $\n_\g(\h)$ is generated by nilpotent elements.
More precisely, 
the conclusion holds
for normalisers generated by nilpotent elements
when $G$ is simple of classical type (that is, the root system of $G$ is of $A$--$D$ type) and $p>h+1$.

\item Let $p>2h-2$ for the connected reductive group $G$. Then the normalisers $N_G(\h)$ of all subspaces $\h$ of $\g$ are smooth whenever $\n_\g(\h)$ is generated by nilpotent elements.
\end{enumerate}
\end{TB}

\begin{remarks}\label{thmbrems}
(a). The bounds in Theorem B(i) are tight when $G$ is classical of type $A$, $B$ or $C$: whenever $p\leq h+1$ the smallest irreducible representation of the first Witt algebra or its adjoint gives rise to a non-smooth normaliser which satisfies the hypotheses. Theorem B(i) is also tight for $G_2$, as it contains a copy of the Witt algebra as a maximal subalgebra when $p=7$; more generally, the conclusion of Theorem B(i) fails for all exceptional algebraic groups when $p=h+1$ (see \cite{HSMax}).   

(b). Suppose that $k$ is not algebraically closed, and that $G$ is a connected reductive algebraic group defined over $k$ with a closed, $k$-defined subgroup-scheme $H$.
Since smoothness is a geometric property, we have that $N_G(H)$ is smooth if and only if $N_{G_{\overline{k}}}(H_{\overline{k}})$ is smooth. Hence Theorems A and B give
sufficient conditions for the smoothness of normalisers over general base fields.
\end{remarks}

In proving the theorems above we require several auxiliary results which may be of independent interest. The first is necessary in proving Theorem B(i).

\begin{TC}Let $\g=\Lie(G)$ for $G$ a simply-connected classical algebraic group over an algebraically closed field $k$ and let $p>2$ be a very good prime for $G$. Then any maximal non-semisimple subalgebra of $\g$ is parabolic.\end{TC}

\begin{remark}An announcement of a full classification of the maximal non-semisimple subalgebras of the Lie algebras of classical groups is given in \cite{Ten87}. We provide a straightforward proof of the stated part in \S\ref{sec:maxvgood} below.\end{remark}

The proof of Theorem B(i) also uses a number of results on cohomology of low-dimensional modules. Such results have something of a history: in \cite{Jan97} Jantzen proved that a module for a connected reductive algebraic group with $p\geq \dim V$ is completely reducible. Building on this, Guralnick tackled the case of finite simple groups in \cite{Gur99}; this time one needs $p\geq\dim V+2$ for the same conclusion. In a different direction, Serre proved in \cite{Ser94s} that if two semisimple modules $V_1$ and $V_2$ for an arbitrary group satisfy $\dim V_1+\dim V_2<p+2$ then their tensor product is semisimple. Extending work of Bendel--Nakano--Pillen, we add analogues of these results for Lie algebras and Frobenius kernels of reductive algebraic groups tackling the `crucial case' of a question of Serre \cite[Question 1.2]{Ser94s} (though see Footnote \ref{deligne} below). We summarise our results when $G$ is simple into the following. The extensions to the case $G$ is semisimple or reductive can be found in \S\ref{sec:cohom}, where also can be found any unexplained terminology. 

\begin{TD}Suppose $G$ is a simple algebraic group and let $G_r$ be its $r$-th
Frobenius kernel with $\g$ its Lie algebra.
Let $V$ be a $k$-vector space with $\dim V\leq p$. 
\begin{itemize}
\item[(a)] 
Suppose $V$ is a $G_r$-module. Then
$V$ is completely reducible unless $\dim V=p$, and either $G$ is of type $A_1$ or $p=2$ and $G$ is of type $C_n$. In the exceptional cases, $V$ is known explicitly.

\item[(b)] Suppose $\g=[\g,\g]$ and $V$ is a $\g$-module. Then either $V$ is completely reducible or $\dim V=p$, $G$ is of type $A_1$ and $V$ is known explicitly.

\item[(c)] Let $p>h$. Then $\opH^2(\g,L(\mu))=0$, for all $\mu$ in the lowest alcove $C_\Z$, unless $G$ is of type $A_1$ and $\mu=(p-2)$; or $G$ is of type $A_2$ and $\mu=(p-3,0)$ or $(0,p-3)$.

\item[(d)] Suppose
$V$ and $W$ are semisimple $\g$-modules with $\dim V+\dim W<p+2$. Then $V\otimes W$ is semisimple and $\opH^2(\g,V\otimes W)=0$.
\footnote{\label{deligne}In \cite{Del13}, it is proved that if $V$ and $W$ are semisimple modules for an affine group scheme, satisfying $\dim V+\dim W<p+2$, then $V\otimes W$ is semisimple. The semisimplicity part of Theorem D(d) can be deduced from this. In fact, in part of Corollary \ref{SerCor} we do prove the semisimplicity statement for arbitrary Lie algebras. Our proof is different to Deligne's, relying just on a theorem of Strade, together with Theorem C.}\end{itemize}
\end{TD}

We also mention a further tool, used in the proofs of Theorems A and B(i), for which we need a definition due to Richardson: 
Suppose that $(G',G)$ is a pair of reductive algebraic groups such that $G \subseteq G'$ is a closed
subgroup.
We say that $(G',G)$ is a \emph{reductive pair} 
provided there is a subspace $\m \subseteq \Lie(G')$ such that $\Lie(G')$ decomposes as a $G$-module into a direct sum $\Lie(G')=\Lie(G) + \m$.
Adapting a result from \cite{Her13} we show
\begin{PF}Let $(G',G)$ be a reductive pair and let $H \leq G$ be a
closed subgroup scheme.
Then if $N_{G'}(H)$ is smooth, $N_{G}(H)$ is smooth too.
\end{PF}

\subsection*{Acknowledgements}
Many thanks to Cornelius Pillen and Brian Parshall who provided a number of helpful hints during the production of this article, and to Alexander Premet for some useful conversations. Also we would like to express our gratitude to the referee for their comments on the paper and encouraging us to elucidate a couple of our arguments along the lines they suggested.
The first author acknowledges financial
support from ERC Advanced Grant 291512 (awarded to Gunter Malle).

\section{Notation and preliminaries}
Let $k$ be a field of characteristic $p\geq 0$ and let
$G$ be an algebraic group defined over $k$. 
Unless otherwise noted, $k$ will assumed to be algebraically closed.
For all aspects to do with the representation theory of
a connected reductive algebraic group $G$ we keep notation compatible with \cite{Jan03}. In particular, $R$ is the root system of $G$, and $h$ is the associated Coxeter number.

For a closed subgroup $H \leq G$,
we consider the scheme-theoretic normaliser $N_G(H)$, respectively
centraliser $C_G(H)$ of $H$ in $G$. We define $N_G(H)$ to be  subfunctor of $G$ which takes a $k$-algebra $A$ and returns the subgroup of elements 
\[N_G(H)(A)=\{g\in G(A) : gH(B)g^{-1}=H(B)\}\] for all $A$-algebras $B$. Similarly, the centraliser is defined via 
\[C_G(H)(A)=\{g\in G(A) : gh=hg \text{ for all }h\in H(B)\}.\]

Since $H$ is closed, $N_G(H)$ and $C_G(H)$ are closed subgroup schemes of $G$.

By contrast, for any affine algebraic group $H$ over $k$, we denote 
by $H_\red$ the smooth subgroup with $k$-points $H_\red(k)=H(k)$. As $k$ is algebraically closed, the existence and uniqueness of such a subgroup is explained for example in \cite[Prop.~5.1]{JSM} and (as we will use in the sequel) we have that $N_G(H)_\red(k')=N_{G(k')}(H(k'))(k')$ (resp.\  $C_G(H)_\red(k')=C_{G(k')}(H(k'))(k')$) by \cite[\S VII.6]{JSM}.

Let $\g$ be a Lie algebra over $k$. When the characterstic of $k$ is greater than $0$, $\g$ is often referred to as a modular Lie algebra, and as such our reference for the theory is \cite{SF88}. Recall that a Lie algebra $\g$ is \emph{semisimple} if its solvable radical is zero, and that in characteristic $p>0$ this is not enough to ensure that it is the direct sum of simple Lie algebras. 

Sometimes but not all the time, we will have $\g=\Lie(G)$ for $G$ an algebraic group, in which cas we refer to $\g$ as \emph{algebraic};
in this case, $\g$ will carry the structure of a restricted Lie algebra. Bear in mind that $\Lie(G)$ may not be semisimple even when $G$ is. Examples of this sort only occur in not-very-good characteristic; for instance, $\sl_2=\Lie(\SL_2)$ in characteristic $2$ gives a restricted structure on the solvable Lie algebra $\sl_2$ with $1$-dimensional centre.

More generally, all restricted Lie algebras are of the form $\Lie(H)$, where
$H$ is an infinitesimal group scheme of height one over $k$. 
Under this correspondence, the restricted subalgebras of $\g=\Lie(G)$ correspond
to height one subgroup schemes of $G$.
If the centre $Z(\g)=0$, then a Lie algebra $\g$ has at most one restricted structure.
In particular, if two semisimple restricted Lie algebras are isomorphic as 
Lie algebras, they are isomorphic as restricted Lie algebras. 

An abelian $p$-subalgebra $\h$ of $\g$ consisting of semsimple elements is called a \emph{torus} of $\g$. Cartan subalgebras of algebraic Lie algebras are always toral and in fact the Lie algebras of maximal tori of the associated algebraic group. This follows from \cite[Thm.\ 13.3]{Hum67}.

If $\g$ is a restricted Lie algebra, a representation $V$ is called
\emph{restricted} provided it is given by a morphism of restricted Lie
algebras $\g \to \gl(V)$. The following fact follows e.g.\ from
the Kac--Weisfeiler conjecture (see \cite[Cor. 3.10]{Premet95}):
if $G$ is a simple algebraic group defined in very good
characteristic, and if $V$ is an irreducible $\g$-module with $\dim V < p$,
then $V$ is restricted. In particular, it is well-known that $V$ is then obtained by
differentiating a simple restricted rational representation of $G$.

When $\g$ is a Lie algebra, $\Rad(\g)$ is the solvable radical of $\g$ and $N(\g)$ is the nilradical of $\g$. If $\g\subseteq\gl(V)$ there is also the radical of $V$-nilpotent elements $\Rad_V(\g)$. When $\g$ is restricted, $\Rad_p(\g)$ is the \emph{$p$-radical} of $\g$, defined to be the biggest $p$-nilpotent ideal.
Further, $\g$ is \emph{$p$-reductive} if the radical $\Rad_p(\g)$ is zero.
Recall the following properties 
 from \cite[\S 2.1]{SF88}:
\begin{lemma}
\begin{itemize}
\item[(a)] $\Rad_p(\g)$ is contained in the nilradical $N(\g)$ and hence
in the solvable radical of $\g$. In particular,
semisimple Lie algebras are $p$-reductive.
\item[(b)] $\Rad_p(\g)$ is the maximal $p$-nil (that is, consisting
of $p$-nilpotent elements) ideal of $\g$.
\item[(c)] $\g/\Rad_p(\g)$ is $p$-reductive.
\end{itemize}
\end{lemma}

In particular, by part (b), if $\g \subseteq \gl(V)$ is a restricted subalgebra
then $\Rad_p(\g) = \Rad_V(\g)$.
If $\g\subseteq\gl(V)$ is a  restricted Lie subalgebra and $G_1$ is the height one subgroup scheme of $\GL(V)$ associated to $\g$, then $\g$ is $p$-reductive if and only if
$G_1$ is reductive in the sense that is has no connected normal nontrivial unipotent 
subgroup schemes.
For the usual notion of reductivity of smooth algebraic groups only smooth
unipotent subgroups are considered. The relation between these two concepts
is as follows:

\begin{prop}[\!\!{\cite{Vas05}}] \label{vasiu}
Let $G$ be a connected reductive algebraic group.
Then $G$ has no non-trivial connected normal unipotent subgroup schemes,
except if both
$p=2$ and $G$ contains a direct factor isomorphic to $\SO_{2n+1}$ for some $n\geq1$.
\end{prop}

Since there are a number of possible definitions, let us be clear on the following: We define a \emph{Borel subalgebra} (resp.\ \emph{parabolic subalgebra}, resp.\ \emph{Levi subalgebra}) of $\g$ to be $\Lie(B)$ (resp.\ $\Lie(P)$, resp.\ $\Lie(L)$), where $B$ (resp.\ $P$, resp.\ $L$) is a Borel (resp.\ parabolic, resp.\ Levi subgroup of a parabolic) subgroup of $G$.

By $P=LQ$ we will denote a parabolic subgroup of $G$ with unipotent radical $Q$ and Levi factor $L$. We will usually write $\p=\Lie(P)=\l+\q$. A fact that we will use continually during this paper, without proof, is that if $H$ (resp. $\h$) is a subgroup (resp. subalgebra) of $P$ (resp. $\p$), such that the projection to the Levi is in a proper parabolic of the Levi, then there is a strictly smaller parabolic $P_1<P$ (resp. $\p_1<\p$) such that $H\leq P_1$ (resp. $\h\leq \p_1$). See \cite[Prop.\ 4.4(c)]{BT65}.

We also use the following fact: If $\t \subseteq \gl_n$ is a torus, then $C_{\GL_n}(\t)$
is a Levi subgroup (this follows e.g.\ from the construction
of a torus $T \subseteq \GL_n$ in \cite[Prop.\ 2]{Dieu:1953} with $C_{\GL_n}(\t) = C_{\GL_n}(T)$). 

Let $V$ be an $\g$-module and let $\lambda:V\times V\to k$ be a bilinear form on $V$. We say $\g$ preserves $\lambda$ if $\lambda(x(v),w)=-\lambda(v,x(w))$ for all $x\in \g$, $v,w\in V$. 

We recall definitions of the algebraic simple Lie algebras of classical type: those with root systems of types A--D. Then $\mathfrak{o}(V)$ is the set of elements $x\in \gl(V)$ preserving the form $\lambda(v,w)=v^tw$. $\mathfrak{so}(V)$ is the subset of traceless matrices of $\mathfrak{o}(V)$. On the other hand  when $\dim V$ is even, $\sp(V)$ is the set of elements preserving the form $\lambda(v,w)=v^tJw$ with $J=[[0,-I_n],[I_n,0]]$. If char $k\neq 2$ then $\sp(V)$ and $\so(V)$ are simple (see below).

We say $\sp(V)$ is of type $C_n$ with $2n=\dim V$; $\so(V)$ is of type $B_n$ when $\dim V=2n+1$, or type $D_n$ when $\dim V=2n$. One fact that we shall use often in the sequel is that that for types B--D, parabolic subalgebras are the stabilisers of totally singular subspaces. (See for example, \cite{Kan79}.)

Furthermore recall that if $G$ is simple, then $\g$ is simple at least whenever $p$ is very good. See \cite[Cor.~2.7]{Hog82} for a more precise statement. This means in particular that $\sl(V)$ is simple unless $p|\dim V$, in which case the quotient $\psl(V)=\sl(V)/kI$ is simple; we refer to such algebras as type $A_n$ classical Lie algebras, where $\dim V=n+1$. In all cases, we refer to $V$ as the \emph{natural module} for the algebra in question.

We make extensive use of the current state of knowledge of cohomology in this paper,
especially in \S\ref{sec:cohom}.
Importantly, recall that the group $\Ext^1_A(V,W)$
(with $A$ either an algebraic group or a Lie algebra)
corresponds to the equivalence classes of extensions $E$ of $A$-modules $0\to W\to E\to V\to 0$,
and that $\opH^2(A,V)$ measures the equivalence classes of central extensions $B$ of $V$ by $A$,
equivalence classes of exact sequences  $0\to V\to B\to A\to 0$,
where $B$ is either an algebraic group or a Lie algebra.
We remind the reader that for restricted Lie algebras, two forms of cohomology are available---the ordinary Lie algebra cohomology, denoted $\opH^i(\g,V)$ or the restricted Lie algebra cohomology (where modules respectively morphisms are assumed to be restricted). 
Since the latter can always be identified with 
$\opH^i(A,V)$ for $A$ the height one group scheme
associated to $\g$, we shall always use the associated group scheme
when we wish to discuss restricted cohomology.

Finally, we record the following theorem of Strade which is a central tool
in our study of small-dimensional representations. Let char $k=p>0$ and let $O_1=k[X]/X^p$ be the truncated polynomial algebra. Then the first Witt algebra $W_1$ is the set of derivations of $O_1$, with basis $\{X^r\partial\}_{0\leq r\leq p-1}$, where $\partial$ acts on $O_1$ by differentiation of polynomials. For $p>2$, $W_1$ is simple, and for $p>3$, $W_1$ is not the Lie algebra of any algebraic group. Since there is a subspace $k\leq O_1$ fixed by $W_1$, we see that $W_1$ has a faithful $(p-1)$-dimensional representation for $p>2$. 
\begin{theorem}[\!\!{\cite[Main theorem]{Str73}}]\label{Strade}
Let $\g$ be a semisimple Lie subalgebra of $\gl(V)$ over an algebraically closed field $k$ of characteristic $p>2$ with $p>\dim V$. 
Then $\g$ is either a direct sum of algebraic Lie algebras or $p=\dim V+1$ and $\g$ is the $p$-dimensional Witt algebra $W_1$.
\end{theorem}

\section{Smoothness of normalisers of reductive subgroups}

Let $G$ be a connected reductive algebraic group and let $T$ be a maximal torus in $G$ with associated roots $R$, coroots $R^\vee$, characters $X(T)$ and cocharacters $Y(T)$.
We say that a prime $p$ is \emph{pretty good} for $G$ provided it is good for $R$ and provided that both $X(T)/\Z R$ and $Y(T)/\Z R^\vee$ have no $p$-torsion.
We recall the main result of \cite{Her13}.

\begin{prop} \label{prop:pretty}
Let $G$ be as above, and let $p=\Char(k)$. Then $p$ is pretty good for $G$ if and only if all centralisers of closed subgroup schemes in $G$ are smooth.
\end{prop}

\begin{theorem}\label{thmreductivenorms} Let $G$ be a connected reductive algebraic group. Then the normalisers $N_G(H)$ of all (smooth) connected reductive subgroups are smooth if $p$ is a pretty good prime for $G$.\end{theorem}
\begin{proof}
Let $H \leq G$ be a closed, connected reductive subgroup of $G$. We have an exact sequence of group functors
\begin{align*}
1 \rightarrow C_G(H) \rightarrow N_G(H) \xrightarrow{\Int} \Aut(H).
\end{align*}
Here the first map is the natural inclusion, the second map maps $x \in G$ to the automorphism $\Int(x)$ of $H$ given by conjugation with $x$,
and $\Aut(H)$ is the group functor that associates to each $k$-algebra $S$ the group of automorphisms of the group scheme $H_S$.
By \cite[XXIV, Cor.\ 1.7]{SGA3}, we have that $\Aut(H)^0 = \Int(H)$ is smooth, which implies that $\Int(N_G(H))$ is smooth.
By Proposition \ref{prop:pretty}, $C_G(H)$ is smooth. Thus the outer terms in the exact sequence of affine group schemes
\begin{align*}
1 \rightarrow C_G(H) \rightarrow N_G(H) \rightarrow \Int(N_G(H)) \rightarrow 1
\end{align*}
are smooth, which forces $N_G(H)$ to be smooth.
\end{proof}

\begin{remark}The implication in the theorem cannot quite be reversed.
For example if $G$ is $\SL_2$, $p=2$ is not pretty good, but a connected reductive subgroup is either trivial, or a torus, whose normaliser is smooth. However, we give examples of non-smooth normalisers of connected reductive subgroups in bad characteristics in Examples \ref{badchar} below.\end{remark}

\section{On exponentiation and normalising, and the proof of Theorem B(ii)}

Let $G$ be a connected reductive group. We recall the existence of exponential and logarithm
maps for $p$ big enough, see \cite[Thm.\ 3]{Ser98} or \cite[Prop.~5.2]{Sei00}.
We fix a maximal torus $T$ and a Borel subgroup $B = T \ltimes U$ containing $T$.

\begin{theorem} \label{thm:exp}
Assume that $p > h$ ($p \geq h$ for $G$ simply connected), where $h$ is the Coxeter number of $G$.
Then there exists a unique isomorphism of varieties
$\log:G^u \rightarrow \g_\text{nilp}$,
whose inverse we denote by
$\exp:\g_\text{nilp} \rightarrow G^u$,
with the following properties:
\begin{itemize}
\item[(i)] $\log \circ \sigma  = d\sigma \circ \log$ for all $\sigma \in \Aut(G)$;
\item[(ii)] the restriction of $\log$ to $U$ is an isomorphism of algebraic groups
$U \rightarrow \Lie(U)$, whose tangent map is the identity; here the group law on
$\Lie(U)$ is given by the Hausdorff formula;
\item[(iii)] $\log(x_\alpha(a)) = aX_\alpha$ for every root $\alpha$ and $a \in k$,
where $X_\alpha = dx_\alpha(1)$.
\end{itemize}
\end{theorem}

The uniqueness implies that for $G = \GL(V)$, $p\geq \dim V$,
$\exp$ and $\log$ are the usual truncated series.

Recall (cf.\ \cite{Ser98}) that for a $G$-module $V$, the number $n(V)$ is defined as 
$n(V) = \sup_\lambda n(\lambda)$, where $\lambda$ ranges over all $T$-weights of $V$,
and where $n(\lambda) = \sum_{\alpha \in R^+} \langle \lambda, \alpha^\vee \rangle$. 
For the adjoint module $\g$, one obtains $n(\g) = 2h-2$.

\begin{prop} \label{prop:expad}
Let $\rho:G \rightarrow \GL(V)$ be a rational representation of $G$.
Suppose that $p >h$ and $p>n(V)$. Let $x \in \g$ be a nilpotent element.
Then
\begin{align*}
\rho(\exp_G x) = \exp_\GL ( d\rho(x)).
\end{align*}
In particular, if $p > 2h-2$, then 
$\Ad(\exp_G x) = \exp_\GL(\ad(x))$. 
\end{prop}

\begin{proof}
Consider the homomorphism $\varphi: \Ga \rightarrow \GL(V)$
given by $\varphi(t) = \rho(\exp_G(t.x))$.
Under our assumptions, it follows  from \cite[Thm.\ 5]{Ser98}
that $\varphi$ is a morphism of degree $<p$,
(i.e.\ the matrix entries
of $\varphi$ are polynomials of degree less than $p$ in $t$).
Moreover, $d \varphi(1) = d\rho(x)$.
By \cite[\S 4]{Ser94s}, this implies that $d \rho(x)^p = 0$ and
that $\varphi$ agrees with the homomorphism
$t \mapsto \exp_\GL(t.d\rho(x))$.
The claim follows.
\end{proof}

\begin{lemma}\label{exp} Let $X\in\gl(V)$ be a nilpotent element 
satisfying $X^n=0$ for some integer $n\leq p$.
Let $l,r \in \End(\gl(V))$ be left multiplication with $X$,
respectively right multiplication with $-X$.
Set $W = W_p(l,r) \in \End(\gl(V))$, where $W_p(x,y)$ is the 
the image of $\frac{1}{p}((x+y)^p-x^p-y^p) \in \Z[x,y]$ in
$k[x,y]$. Let $\h$ be a subset of $\gl(V)$ normalised (resp.\ centralised) by $X$.
Suppose that $\h \subseteq \ker(W)$.
Then $\exp(X) \in \GL(V)$ normalises (resp.\ centralises) $\h$.

In particular, if $p \geq 2n-1$, then $W=0$ and so $\exp(X)$ normalises (resp.\ centralises)
every subspace that is normalised (resp.\ centralised) by $X$.
\end{lemma}

\begin{proof}
Since the nilpotence degree of $X$ is less than $p$,
the exponential $\exp(X)=1+X+X^2/2+\dots$ gives a well-defined element of $\GL(V)$. 
Moreover, for each $Y \in \h$ we have the equality 
$$\Ad(\exp(X))(Y) \kern-1pt = \kern-1pt \exp(\ad(X))(Y) \kern-1pt = \kern-1pt Y+\ad(X)(Y) +\ad(X)^2(Y)/2 +\dots \in \gl(V).$$ 

Indeed, we have $\ad(X) = l + r$, 
and $\Ad(\exp(X)) = \exp(l)\exp(r)$. 
Now by \cite[(4.1.7)]{Ser94s}, $\exp(l)\exp(r) = \exp(l+r-W)$.
Since $l$ and $r$ commute with $W$, we deduce
$(l + r - W)^m (Y) = (l+r)^m(Y)$ for each $m \geq 0$.
Thus
$\Ad(\exp(X))(Y) = \exp(l+r)(Y) = \exp(\ad(X))(Y)$, as claimed.
Hence $\exp(X)$ is contained in $N_{\GL(V)}(\h)$ whenever
$X\in\n_{\gl(V)}(\h)$ and $\exp(X)\in C_{\GL(V)}(\h)$ whenever $X\in\c_{\gl(V)}(\h)$.

Moreover, $W_p(l,r) = \sum_{i=1}^{p-1} c_i l^i r^{p-i}$ for certain non-zero
coefficients $c_i \in k$. In particular, this expression vanishes for $p \geq 2n-1$.
\end{proof}

\begin{cor}\label{expParab}
Let $\p = \q + \l \subseteq \gl(V)$ be a parabolic subalgebra,
and suppose that $p \geq \dim V$.
If $X \in \q$ normalises a subset $\h \subseteq \p$,
then so does $\exp(X)$.
\end{cor}

\begin{proof}
By Lemma \ref{exp}, it suffices to show that $\p \subseteq \ker(W)$.
Let $0 = V_0 \subseteq V_1 \subseteq \dots \subseteq V_m = V$ be a flag with
the property
\begin{align*}
\p &= \{ Y \in \gl(V) \mid Y V_i \subseteq  V_i \} \\
\q &= \{ Y \in \gl(V) \mid Y V_i \subseteq  V_{i-1} \}.
\end{align*}
By assumption, we have $p \geq m$,
and therefore all products $X_1 \dots X_{p+1}$ with all $X_i \in \p$
and all but one $X_i \in \q$ vanish on $V$. 
In particular $l^ir^{p-i}(Y)=0$ for all $Y \in \p$ and hence $W(Y)=0$.  
\end{proof}

\begin{lemma}\label{genbynilp}Suppose $\g$ is a subalgebra of $\gl(V)$ generated as a $k$-Lie algebra by a set of nilpotent elements $\{X_i\}$ of nilpotence degree less than $p$, and let $G=\overline{\langle\exp(t.X_i)\rangle}$ be the closed subgroup of $\GL(V)$ generated by $\exp(t.X_i)$ for each $t\in k$. Then $\g\leq\Lie(G)$.\end{lemma}

\begin{proof} Since $\Lie(G)$ contains the element $d/dt \exp(t.X_i)|_{t=0}$ it contains each element $X_i$. Since $\g$ is generated by the elements $X_i$, we are done.\end{proof}

\begin{proof}[Proof of Theorem B(ii)]
Let $\h$ be a subspace of $\g$ and let $\n = \n_\g(\h)$ be
the Lie-theoretic normaliser of $\h$ in $\g$.

Let $\{x_1,\dots,x_r\}$ be a set of nilpotent elements generating $\n$.
To show that $N_G(\h)$ is smooth, it suffices to show that each 
$x_i$ belongs to the Lie algebra of $N_G(\h)_\red$.

But for a nilpotent generator $x_i$, we may consider the smooth closed subgroup
$M_i = \overline{ \langle \exp(t.x_i) \mid t \in k \rangle}$ of $G$.
By Proposition \ref{prop:expad}, $M_i \subseteq N_G(\h)_\red$ 
and hence $x_i \in \Lie(M_i) \subseteq \Lie(N_G(\h)_\red)$, as required. 
\end{proof}

\section{Reductive pairs: Proof of Proposition E}\label{sec:reducpairs}

The following definition is due to Richardson \cite{Ric67}.
\begin{defn}
Suppose that $(G',G)$ is a pair of reductive algebraic groups such that $G \subseteq G'$ is a closed
subgroup.
Let $\g'=\Lie(G)$, $\g=\Lie(G)$.
We say that $(G',G)$ is a \emph{reductive pair} 
provided there is a subspace $\m \subseteq \g'$ such that $\g'$ decomposes as a $G$-module into a direct sum $\g'=\g \oplus \m$.
\end{defn}

With $p$ sufficiently large, reductive pairs are easy to find.

\begin{lemma}[\!\!{\cite[Thm. 3.1]{BHMR11}}]
\label{lem:redpair}
Suppose $p>2\dim V-2$ and $G$ is a connected reductive subgroup of $\GL(V)$. Then $(\GL(V),G)$ is a reductive pair.\end{lemma}

We need a compatibility result for normalisers of subgroup schemes of height one.

\begin{lemma}\label{ngh}
Let $H \subseteq G$ be a closed subgroup scheme of height one, with $\h = \Lie(H)$.
Then $N_G(H) = N_G(\h)$ (scheme-theoretic normalisers).
\end{lemma}

\begin{proof}
We have a commutative diagram 
$$
\begin{CD}
\Hom(H,H) @>>> \Hom_{p-\text{Lie}}(\h,\h) \\
@VVV @VVV \\
\Hom(H,G) @>>> \Hom_{p-\text{Lie}}(\h,\g),
\end{CD}
$$
where the horizontal arrows are given by differentiation and are bijective
(cf.\ \cite[II, \S7, Thm.\ 3.5]{DG:1970}).
Now if $x \in N_G(\h)$, the map $\Ad(x)_\h$ in the bottom right corner may be lifted via the top right corner
to a map in $\Hom(H,H)$. The commutativity of the diagram shows that 
conjugation by $x$ stabilises $H$,
and hence $x \in N_G(H)$. This works for points $x$ with values
in any $k$-algebra, and hence proves the containment
of subgroup schemes $N_G(\h) \subseteq N_G(H)$. The reverse inclusion is clear.
\end{proof}

We show that the smoothness of normalisers descends along reductive pairs. Let us restate and then prove Propostion E.

\begin{prop}\label{redpairnorm}
Let $(G',G)$ be a reductive pair and let $H \subseteq G$ be a closed subgroup scheme.
If $N_{G'}(H)$ is smooth, then so is $N_G(H)$.

In particular, if $\h \subseteq \g$ is a restricted subalgebra
and if $N_{G'}(\h)$ is smooth, then so is $N_G(\h)$.
\end{prop}

\begin{proof}
The last assertion follows from Lemma \ref{ngh}.

Let $H \subseteq G$ be a closed subgroup scheme.
We follow the proof of \cite[Lem.\ 3.6]{Her13}.
Let $\g' = \g \oplus \m$ be a decomposition of $G$-modules.

By \cite[II, \S5, Lem.\ 5.7]{DG:1970}, we have
\begin{align*}
\dim \Lie(N_{G'}(H)) &= \dim \h + \dim (\g'/\h)^H 
= \dim \h + \dim (\g/\h)^H + \dim \m^H \\
&= \dim \Lie(N_G(H)) + \dim \m^H \geq \dim N_G(H) + \dim \m^H.
\end{align*}

On the left hand side, as $N_{G'}(H)$ is smooth by assumption,
we have $\dim N_{G'}(H) = \dim \Lie(N_{G'}(H))$. 
Thus to show that $N_G(H)$ is smooth,  
it suffices to show that $\dim N_{G'}(H) - \dim N_G(H) \leq \dim \m^{H}$.

Now as in \cite[Lem.\ 3.6]{Her13}, one shows that there is a monomorphism of quotient schemes
$N_{G'}(H)/N_G(H) \hookrightarrow (G'/G)^H$,
and that the tangent space on the right hand side identifies
as $T_{\bar e}(G'/G)^H \cong \m^H$.
The claim follows.
\end{proof}

\section{Lifting of normalising tori and the proof of Theorem A}\label{sec:aii}

In this section we let $G=\GL(V)$ and $\h$ be a subspace of $\g$. We would like to lift a normaliser $\n_\g(\h)$ to a subgroup $N$ normalising $\h$ such that $\Lie(N)=\n_\g(\h)$. It turns out that the hardest part of this is to find a lift of a maximal torus normalising $\h$. This is the content of the next lemma.

\begin{lemma} \label{findingtorus}
Let $G=\GL_n$ with $p>2^{2n}$ and let 
$\h \subseteq \g$ be any subspace of $\g=\Lie(G)$.
Suppose that $\c \subseteq \g$ is a torus normalising $\h$.
Then $\c = \Lie(C)$ for a torus $C \subseteq N_G(\h)$.
\end{lemma}

\begin{proof}
Let $T$ be a diagonal maximal torus of $\GL_n$ and $\t=\Lie(T)$. Since $\c$ consists of semisimple elements, we may assume $\c\subseteq\t$.

Since $\c$ is restricted, it has a basis defined over $\F_p$ of elements
$Z_1,\dots,Z_s$ with $Z_i=\diag(z_{i1},\dots,z_{in})$ and 
each $z_{ij}\in\F_p$. By \cite[Prop.\ 2]{Dieu:1953} we may assume that $\c$ is a maximal torus of $\n_\g(\h)$, which we do from now on. 

Since $k$ is algebraically closed,
we may take a decomposition of $\h$ into weight spaces for $\c$. 
We have $\h=\h_0\oplus \bigoplus_\alpha \h_\alpha$ where $\h_0$ is some set of elements
commuting with $\c$, $\alpha$ is a non-trivial linear functional 
$\c\to k$ and each $\h_\alpha$ is a subspace of $\gl_n$ with $[c,X]=\alpha(c)X$ for $c\in \c$ and $X\in \h_\alpha$.

Let $\{X_i\}$ be a basis for $\h$ with each $X_i\in\h_0$ or $\h_\alpha$ for some $\alpha$ as above. Then $\c=\bigcap_i\n_\t(\la X_i\ra)$. Suppose $c=\diag(c_1,\dots,c_n)$. The condition $c\in\n_\t(\la X_i\ra)$ puts a set of conditions on the $c_i$. If only one entry of the matrix $X_i$ is non-zero or $X_i$ is diagonal, then $\t$ normalises $X_i$, hence the set of conditions is empty. Otherwise, if $(X_i)_{j,k}$ and $(X_i)_{l,m}$ are non-zero, then $c$ normalising $\la X_i\ra$ implies $c_j-c_k=c_l-c_m$. Letting $\mathbf c= (c_1,\dots,c_n)$ this condition can be rewritten as a linear equation $\mathbf r\mathbf c=0$, where $\mathbf r$ is an appropriate row vector whose entries are all $0$, except for up to four, where the non-zero entries take the values, up to signs or permutations, $(1,-1),(2,-2),(1,-2,1)$ or $(1,-1,-1,1)$ according to the values of $j,k,l$ and $m$. The collection of these, say $m$ relations, across $i$ and all pairs of non-zero entries in $X_i$  gives an $m\times n$ integral matrix $R$ so that $c\in \c$ if and only if it satisfies the equation $R\mathbf c=0$ modulo $p$. Similarly, if $\chi(t)=\diag(t^{a_1},\dots,t^{a_n})$ is a cocharacter with image in $T$, then one checks that $\chi(t)$ normalises $\h$ if the integral equation $R\mathbf a=0$ where $\mathbf a=(a_1,\dots,a_n)$.
If the nullity of $R$ is the same modulo $p$ as it is over the integers then for any $c\in \n_\t(\h)$,
there exists a cocharacter $\chi$ of $N_T(\h)$ with $d/dt|_{t=1}(\chi(t))=c$ and we are done.
But if the nullity of $R$ modulo $p$ differs from the nullity of $R$ over the integers, then we must have that $p|d_i$ for $d_i$ one of the non-zero elementary divisors of $R$.
Now by the theory of Smith Normal Form, if $r\in \NN$ is taken maximal so that there exists a
non-vanishing $r\times r$ minor, then the elementary divisors of $R$ are all at most the greatest common divisor of all non-zero $r\times r$ minors. Let $M$ be such an $r\times r$ minor. We are going to argue by induction on $r$ that $|\det(M)|\leq 2^{2r}$. Since $r\leq n$, the hypothesis will then show that $p$ is not a prime factor ofÊ $\det(M)$, as required.
%Let us bound the prime factors of the determinant of an $r\times r$ minor, $M$ say.

We must have $r\leq n$. If there is a row of $M$ containing only elements of modulus $2$, then at most $2$ of these are non-zero and $2$ is a prime factor of $\det M$; Laplace's formula implies that the remaining matrix has determinant at most $2\det M'$ where $M'$ is a certain $r-1\times r-1$ minor of $M$, so that we are done by induction.  If there are no entries of modulus $2$, then each row contains at most $4$ entries of modulus $1$ and Laplace's formula then implies that $\det M\leq 4\det M'$ where $M'$ is a certain $r-1\times r-1$ minor of $M$ of the required form, so that we are done again by induction. Otherwise there is at least one row with non-zero entries $(1,-2)$ or $(1,-2,1)$. By Laplace's formula and induction, it is now easy to see that $|\det M|\leq 2^{2n-2}+2.2^{2n-2}+2^{2n-2}=2^{2n}$ and we are done.
\end{proof}

We are now in a position to prove Theorem A.

\begin{proof}[Proof of Theorem A]
First consider the case $G=\GL_n$. Let $\h$ be a subspace of $\g$ and let $\n = \n_\g(\h)$ be the Lie-theoretic normaliser of $\h$ in $\g$.

As before, by definition, $\n$ is a restricted subalgebra of $\g$.
Hence, applying the Jordan decomposition for restricted Lie algebras,
we see that $\n$ is generated by its nilpotent and semisimple
elements. Let $\{x_1,\dots,x_r,y_1,\dots,y_s\}$ be such a generating
set with $x_1,\dots,x_r$ nilpotent and $y_1,\dots,y_s$ semisimple.
To show that $N_G(\h)$ is smooth, it suffices to show that all the elements
$x_i$ and $y_j$ belong to the Lie algebra of $N_G(\h)_\red$.

For a nilpotent generator $x_i$, of nilpotence degree at most $n<p$, consider the smooth closed subgroup
$M_i = \overline{ \langle \exp(t.x_i) \mid t \in k \rangle}$ of $G$.
Since $p>2h-2$, we may apply Proposition \ref{prop:expad}, to obtain $M_i \subseteq N_G(\h)_\red$ and hence $x_i \in \Lie(M_i) \subseteq \Lie(N_G(\h)_\red)$, as required. 

It remains to consider the semisimple generators $y_i$.
Let $\t_i := \langle y_i \rangle_p \leq \n$ be the torus
generated by the $p$-powers of $y_i$.
By hypothesis, $p>2^{2n}$ and so we may apply Lemma \ref{findingtorus} to find a torus $T_i \leq N_G(\h)$ 
such that $\Lie(T_i) = \t_i$.
In particular $y_i \in \Lie(N_G(\h)_\red)$.
This finishes the proof in the case $G=\GL(V)$.

If $G$ is a reductive algebraic group suppose $G \rightarrow \GL(V) \cong \GL_d$ is a minimal
faithful module for $G$. 
Now since $p>2^{2\dim V}$, we have that normalisers of all subspaces of $\GL(V)$ are smooth. But now, by Lemma \ref{lem:redpair}, $(\GL(V), G)$ is a reductive pair, so that invoking Proposition \ref{redpairnorm} we obtain that $N_G(\h)$ is smooth. This completes the proof.
\end{proof}

\section{Non-semisimple subalgebras of classical Lie algebras. Proof of Theorem C}\label{sec:maxvgood}
Suppose char $k>2$ for this section.

This section provides proofs for some of the claims made in \cite{Ten87}. Here we tackle the proof of Theorem C.

\begin{prop}[see {\cite[\S5.8, Exercise 1]{SF88}}]
\label{bilinear}Let $\g\leq \gl(V)$ be a Lie algebra acting irreducibly on an $\g$-module $V$ such that $\g$ preserves a non-zero bilinear form. Then $\g$ is semisimple.\end{prop}

\begin{proof}Assume otherwise. Then $\Rad(\g)\neq 0$ and we can find an abelian ideal $0\neq J\triangleleft \g$. Take $x\in J$. As $[x^p,y]=\ad(x)^py\in J^{(1)}=0$, $x^p$ centralises $\g$ and we have that $v\mapsto x^pv$ is a $\g$-homomorphism $V\to V$. Since $k$ is algebraically closed and $V$ is irreducible, Schur's lemma implies that $x^pv=\alpha(x)v$ for some map $\alpha: J \rightarrow k$.

Since $\lambda\neq 0$ there are $v,w$ with $\lambda(v,w)=1$.
Now $\alpha(x)=\lambda(x^pv,w)=-\lambda(v,x^pw)=-\alpha(x)$ so $\alpha(x)=0$. Thus $x^pv=0$ for all $x\in J$. Hence $J$ acts nilpotently on $V$ and so Engel's theorem gives an element $0\neq v\in V$ annihilated by $J$. Since $V$ is irreducible, it follows that $JV=J(\g v)\leq \g Jv=0$. Thus $J=0$ and $\g$ is semisimple.\end{proof}

Since any subalgebra of a classical simple Lie algebra of type $B$, $C$ or $D$ preserves the associated (non-degenerate) form we get

\begin{cor}\label{cor:nonss_red}If $\h$ is a non-semisimple subalgebra of a classical simple Lie algebra $\g$ of type $B,\ C$ or $D$ then $\h$ acts reducibly on the natural module $V$ for $\g$.\end{cor}

\begin{remark}If $\g=\mathfrak{g}_2$ (resp. $\mathfrak{f}_4$, $\mathfrak{e}_7$, $\mathfrak{e}_8$) then a subalgebra acting irreducibly on the self-dual modules $V_7$ (resp. $V_{26}$, or $V_{25}$ if $p=3$, $V_{56}$, $V_{248}=\mathfrak{e}_8$) is semisimple.
Here $V_n$ refers to the usual irreducible module of dimension $n$.
\end{remark}

A subalgebra is \emph{maximal rank} if it is proper
and contains a Cartan subalgebra (CSA) of $\g$.
(Note that CSAs of simple algebraic Lie algebras are tori.)
Call a subalgebra $\h$ of $\g$ an \emph{$R$-subalgebra}
if $\h$ is contained in a maximal rank subalgebra of $\g$.

For the following, notice that if $p|\dim V$ then $\sl(V)$ is not simple, though provided $\sl(V)\neq \sl_2$ in characteristic $2$, the central quotient $\psl(V)$ is simple. Now, a subalgebra $\h$ of $\psl(V)$ is an $R$-subalgebra of $\psl(V)$ if and only if its preimage $\pi^{-1}\h$ under $\pi:\sl(V)\to\psl(V)$ is an $R$-subalgebra. We say $\h$ acts reducibly on $V$ if $\pi^{-1}\h$ does.

\begin{prop}\label{reducibly} Let $\g$ be a simple algebraic Lie algebra of classical type and let $\h\leq \g$ act reducibly on the natural module $V$ for $\g$. Then $\h$ is an $R$-subalgebra unless $\g=\so(V)$ with $\dim V=2n$ with $\h\leq \so(W)\times \so(W')$ stabilising a decomposition of $V$ into two odd-dimensional, non-degenerate subspaces $W$ and $W'$ of $V$.\end{prop}

\begin{proof}Let $V$ be the natural module for $\g$ and let $W\leq V$ be a minimal $\h$-submodule, so that $\h\leq \Stab_\g(W)$. If $\g$ is of type $A$ then $\Stab_\g(W)$ is $\Lie(P)$ for a (maximal) parabolic $P$ of $\SL(V)$. Hence $\h$ is an $R$-subalgebra of $\g$. 

If $\g$ is of type $B,\ C$ or $D$, then consider $U=W\cap W^\perp$; this is the subspace of $W$ whose elements $v$ satisfy $\lambda(v,w)=0$ for every $w\in W$. Since $M$ preserves $\lambda$, this is a submodule of $W$, hence we have either $U=0$ or $U=W$ by minimality of $W$. If the latter, $W$ is totally singular. Thus $\Stab_\g W$ is $\Lie(P)$ for a parabolic subgroup $P$ of the associated algebraic group.

On the other hand, $U=0$ implies that $W$ is non-degenerate. Then $V = W \oplus W^\perp$ is a direct sum of $\h$-modules and we see that $\Stab_\g W$ is isomorphic to \begin{enumerate}\item $\sp_{2r}\times \sp_{2s}$ in case $L$ is of type $C$, $\dim W=2s$ and $2r+2s=\dim V$
\item $\so_{r}\times \so_{s}$ in case $L$ is of type $B$ or $D$, $\dim W=s$ and $r+s=\dim V$.\end{enumerate}

Note that by \cite[VII, \S2, No.\ 1, Prop.\ 2]{Bourb05} the dimensions of the CSA of a direct product is the sum of the dimensions of the CSAs of the factors. In case (i), the subalgebra described has the $(r+s)$-dimensional CSA arising from the two factors.  In case (ii), if $\dim V=2n+1$ is odd then one of $r$ and $s$ is odd. If $r$ is odd then $\so_r$ has a CSA of dimension $(r-1)/2$, and $\so_s$ has a CSA of dimension $s/2$, so that the two together give a CSA of dimension $s/2+(r-1)/2=n$. (Similarly if $s$ is odd.) Otherwise $\dim V=2n$ is even. If $\dim W$ is even then $\Stab_\g W$ contains a CSA of dimension $r/2+s/2=n$. If $\dim W$ is odd then we are in the exceptional case described in the proposition. \end{proof}
\begin{remark} In the exceptional case, note that $\so_{2r+1}\times \so_{2s+1}$ contains a CSA of dimension $r+s$, whereas $\so_{2n+2}$ contains a CSA of dimension $n+1=r+s+1$.\end{remark}

\begin{cor}\label{bcdcor}\label{thmb:bcd}
Let $\g$ be of type $B$, $C$ or $D$.
If $\h$ is a maximal non-semisimple subalgebra of $\g$, then $\h$ is $\Lie(P)$ for $P$ a maximal parabolic of $G$.
In particular, if $\h$ is any non-semisimple subalgebra of $\g$, it is an $R$-subalgebra.
\end{cor}

\begin{proof}Assume otherwise. Then $\h$ fixes no singular subspace on $V$. Suppose $\h$ preserves a decomposition $V=V_1\perp V_2\perp\dots \perp V_n$ on $V$ with $n$ as large as possible, with the $V_i$ all non-degenerate. Then $\h\leq \g_1=\so(V_1)\times\dots\times \so(V_n)$ or $\h\leq \g_1=\sp(V_1)\times\dots\times \sp(V_n)$. Since $\h$ is non-semisimple, the projection $\h_1$ of $\h$ in $\so(V_1)$ or $\sp(V_1)$, say, is non-semisimple. Then Proposition \ref{bilinear} shows that $\h$ acts reducibly on $V_1$. Since $\h$ stabilises no singular subspace, the proof of Proposition \ref{reducibly} shows that $\h$ stabilises a decomposition of $V_1$ into two non-degenerate subspaces, a contradiction of the maximality of $n$.\end{proof}

Let $\h$ be a restricted Lie algebra, $I \leq \h$ an abelian ideal
and $V$ an $\h$-module.
Let $\lambda\in I^*$.
Recall from \cite[\S5.7]{SF88} that $\h^\lambda=\{x\in \h|\lambda([x,y])=0 \text{ for all }y\in I\}$
and $V^\lambda=\{v\in V|x.v=\lambda(x)v \text{ for all } x\in I\}$.

\begin{prop}
Let $\h$ be a non-semisimple subalgebra of $\sl(V)$ with $V$ irreducible for $\h$.
Then $p|\dim V$.
\end{prop}

\begin{proof}
Let $\h$ be as described and
let $I$ be a nonzero abelian ideal of $\h$.
If $\h_p$ denotes the closure of $\h$ under the $p$-mapping, then
by \cite[2.1.3(2),(4)]{SF88}, $I_p$ is an abelian $p$-ideal of $\h_p$.
Thus $\Rad \h_p\neq 0$ and $\h_p$ is non-semisimple.
Hence we may assume from the outset that $\h = \h_p$ is restricted with 
nonzero abelian ideal $I$.

Since $\h$ acts irreducibly on $V$, by \cite[Corollary 5.7.6(2)]{SF88} there exist
$S\in \h^*$, $\lambda\in I^*$ such that
$$V\cong \Ind_{\h^\lambda}^\h(V^\lambda,S).$$

If $\lambda$ is identically $0$ on $I$ then $V^\lambda$ is an $\h$-submodule.
We cannot have $V^\lambda=0$ (or else $V=0$) so $V^\lambda=V$ and $I$ acts trivially on $V$, a contradiction since $I\leq \sl(V)$.

Hence $\lambda(x)\neq 0$ for some $x\in I$.
Suppose $V^\lambda=V$. Then as $x\in \sl(V)$, we have $\tr_V(x)=\dim V\cdot\lambda(x)=0$
and thus $p | \dim V$ and we are done.
If $\dim V^\lambda<\dim V$, then by \cite[Prop. 5.6.2]{SF88} we have
$\dim V=p^{\dim L/L^\lambda}\cdot\dim V^\lambda$. Thus again $p|\dim V$, proving the theorem.
\end{proof}

\begin{cor}\label{thmb:sl}If $p \nmid \dim V$ then any non-semisimple subalgebra $\h$ of $\sl(V)$ acts reducibly on $V$. Hence it is contained in $\Lie(P)$ for $P$ a maximal parabolic of $\SL(V)$. In particular $\h$ is an $R$-subalgebra.\end{cor}

Putting together Corollaries \ref{thmb:bcd} and \ref{thmb:sl}, this completes the proof of Theorem C.

As a first application, the following lemma uses Theorem C to show that $p$-reductive
implies strongly $p$-reductive. 
Recall that a restricted Lie algebra is \emph{strongly $p$-reductive} if it is the direct sum of a central torus and a semisimple ideal.

\begin{lemma}\label{predimpstrong}Let $\h \subseteq \gl_n$ be a subalgebra and let $p>n$.
If $\h$ is $p$-reductive, it is strongly
$p$-reductive.
\end{lemma}
\begin{proof}
Take $\p=\l+\q$ a minimal parabolic subalgebra with $\h\leq \p$. Set $\h_l$ to be the image of $\h$ under the projection $\pi:\p\to \l$.
Since $p>n$, we have $\l\cong \gl(W_1)\times\dots\times\gl(W_s)
\cong \sl(W_1) \times \dots \sl(W_s) \times \z$,
where $\z$ is a torus.
Let $\s_i$ be the projection of $\h_\l$ to $\sl(W_i)$,
and let $\z'$ be the projection of $\h_\l$ to $\z$.
If the projection of $\Rad(\h_\l)$ to $\sl(W_i)$ is non-trivial,
then $\s_i$ is not semisimple.
By Theorem C, $W_i$ is not irreducible for $\s_i$. Thus $\p$ is not minimal subject to containing $\h$, a contradiction, proving that all the $\s_i$ are semisimple.
Moreover, $\z' = Z(\h_\l)$, as the projection of $\z$
to each $\sl(W_i)$ must vanish.
This forces $\h_\l \subseteq \s_1 \times \dots \times \s_s \times Z(\h_\l)$ to be strongly $p$-reductive.
As $\h$ is $p$-reductive, we have that $\pi$ is injective on $\h$, and hence $\h \cong \h_l$
is strongly $p$-reductive.
\end{proof}

\section{Complete reducibility and low-degree cohomology for classical Lie algebras: Proof of Theorem D}\label{sec:cohom}

Let $G$ be a connected reductive algebraic group with root system $R$ and let $G_r\triangleleft G$ be the $r$th Frobenius kernel for any $r\geq 1$. It is well-known that the representation theory of $G_1$ and $\g$ are very closely related. In this section we recall results on the cohomology of small $G_r$-modules and use a number of results of Bendel, Nakano and Pillen to prove that small $G_r$-modules are completely reducible with essentially one class of exceptions. We do this by examining $\Ext^1_{G_r}(L(\lambda),L(\mu))$ for two simple modules $L(\lambda)$ and $L(\mu)$ of bounded dimension or weight. While we are at it, we also get information about $\opH^2(G_1,L(\lambda))$.
In a further subsection, we then go on to use this to prove the analogous statements for $\g$-modules. One crucial difference we notice is with central extensions: $\opH^2(\g,k)$ tends to be zero, whereas $\opH^2(G_1,k)$ is almost always not; c.f. Corollary \ref{mcncor} and Theorem \ref{h1h2vanish}.

All the notation in this section is as  in \cite[List of Notations, p.\ 569]{Jan03}:
In particular, for a fixed maximal torus $T \leq G$, we denote by $R$
the corresponding root system, by $R^+$ a choice of positive roots with corresponding
simple roots $S \subseteq R^+$, by $X(T)_+ \subseteq X(T)$ the dominant weights inside
the character lattice, by $L(\lambda)$ the simple $G$-module of highest weight
$\lambda \in X(T)_+$, by $H^0(\lambda)$ the module induced from $\lambda$ with
socle $L(\lambda)$, by $C_\Z$ (resp.\ 
$\bar C_\Z$) the dominant weights inside the lowest alcove
(respectively, in the closure of the lowest alcove).
If $G$ is simply connected, we write $\omega_i \in X(T)_+$ for the 
fundamental dominant weight corresponding to $\alpha_i \in S = \{\alpha_1,\dots,\alpha_l\}$.

Let us recall some results from \cite{McN02} which show the interplay between the conditions that, relative to $p$, (i) modules are of small dimension; (ii) their high weights are small; and (iii) the Coxeter number is small.

\begin{prop}[\!\!{\cite[Prop.~5.1]{McN02}}]\label{mcn1}Let $G$ be simple and  simply connected, let $L$ be a simple non-trivial restricted $G$-module with highest weight $\lambda\in X(T)_+$ and suppose that $\dim L\leq p$. Then
\begin{enumerate}\item We have $\lambda\in \bar C_\Z$.
\item We have $\lambda\in C_\Z$ if and only if $\dim L<p$.
\item We have $p\geq h$. If moreover $\dim L<p$ then $p>h$.
\item If $R$ is not of type $A$ and $\dim L=p$ then $p>h$. If $p=h$ and $\dim L=p$ then $R=A_{p-1}$ and $\lambda=\omega_i$ with $i\in\{1,p-1\}$.\end{enumerate}\end{prop}

\subsection{Cohomology and complete reducibility for small 
\texorpdfstring{$G_1$}{G1}-modules}

We need values of $\opH^i(G_1,H^0(\mu))$ for $\mu\in\bar C_\Z$ and $i=1$ or $2$. Thus $H^0(\mu)=L(\mu)$.

\begin{prop}\label{mcncor}Let $G$ be simple and simply connected and suppose $L=L(\mu)$ with $\mu\in\bar C_\Z$ and $p\geq 3$. Then:

(i) we have $\opH^1(G_1,L)=0$ unless $G$
is of type $A_1$, $L=L(p-2)$ and in that case $\opH^1(G_1,L)^{[-1]}\cong L(1)$;

(ii) suppose $p>h$. Then we have $\opH^2(G_1,L)=0$ unless: $L=k$ and $\opH^2(G_1,k)^{[-1]}\cong \g^*$; or $G=\SL_3$, with $\opH^2(G_1,L(p-3,0))^{[-1]}\cong L(0,1)$ and $\opH^2(G_1,L(0,p-3))^{[-1]}\cong L(1,0)$.\end{prop}

\begin{proof}
Part (i) is immediate from \cite[Corollary 5.4 B(i)]{BNP02}.
The $A_1$ result is well known. Part (ii) requires some argument.
If $\opH^2(G_1,H^0(\mu))\neq 0$ then since $p>h$ we may assume $\mu\in C_\Z$.
Now, the values of $\opH^2(G_1,H^0(\mu))^{[-1]}$ are known from \cite[Theorem 6.2]{BNP07}.
It suffices to find those that are non-zero for which $\mu\in C_\Z \setminus \{0\}$.
All of these have the form $\mu=w.0+p\lambda$ for $l(w)=2$ and $\lambda\in X(T)_+$.
Now, if $l(w)=2$,
we have $-w.0=\alpha+\beta$ for
two distinct roots $\alpha,\beta\in R^+$
(cf.\ \cite[p.\ 166]{BNP07}).
To have $w.0+p\lambda$ in the lowest alcove,
one needs $\langle w.0+p\lambda+\rho,\alpha_0^\vee \rangle<p$.
Now $\langle p\lambda,\alpha_0^\vee \rangle \geq p$
so $\langle w.0+\rho,\alpha_0^\vee \rangle < 0$.
Thus $m:=\langle \alpha+\beta,\alpha_0^\vee \rangle >h-1$.
Now one simply considers the various cases.
If $G$ is simply-laced, then the biggest value of
$\langle \alpha,\alpha_0^\vee \rangle$ is $2$, when $\alpha=\alpha_0$ and $1$ otherwise,
thus $m>h-1$ implies $h\leq 3$.
Thus we get $G=\SL_3$, and this case is calculated in \cite[Prop.~2.5]{SteSL3}.
If $G=G_2$ we have $m$ at most $5$, giving $h$ at most $5$, a contradiction.
If $G$ is type $B$, $C$ or $F$, then  $m$ is at most $4$,
so $G=\Sp_4$, $p\geq 5$ and this is calculated in \cite[Prop.~4.1]{Ibr12}.
One checks that all $\mu$ such that $\opH^2(G_1,L(\mu))\neq 0$ have $\mu\not\in C_\Z$.
\end{proof}

\begin{remark} All the values of $\opH^2(G_r,H^0(\lambda))^{[-1]}$ are known for all $\lambda$ by \cite[Theorem 6.2]{BNP07} ($p\geq 3$) and \cite{Wri11} ($p=2$). For example, $\opH^2(G_1,k)^{[-1]}\cong \g^*$ also when $G$ is of type $A_1$ and $p=2$. Even for $\lambda=0$ there are quite a few exceptional cases when $p=2$: see \cite[C.1.4]{Wri11}. There are also two exceptional cases for $p=3$, for $A_2$ and $G_2$, see \cite[Theorem 6.2]{BNP07}.\end{remark}

One can go further in the case of $1$-cohomology to include extensions between simple modules:

\begin{lemma}[\!\!{\cite[Corollary 5.4 B(i)]{BNP02}}]\label{bnp} Let $G$ be a simple,
simply connected algebraic group not of type $A_1$. If $p>2$ then $\Ext^1_{G_r}(L(\lambda),L(\mu))=0$ for all $\lambda,\mu\in \bar C_\Z$.\end{lemma}

We will use the above result to show that small $G_r$-modules are completely reducible, but we must first slightly soup it up before we use it. 

\begin{lemma}\label{souped}
Let $G$ be a simple, simply connected algebraic group not of type $A_1$ and $p>2$. 

(i) We have $\Ext^1_{G_r}(L(\lambda)^{[s]},L(\mu)^{[t]})=0$ for all $\lambda,\mu\in \bar C_\Z$ and $s,t\geq 0$.

(ii) For $\lambda,\mu\in X_r(T)$, let $\lambda=\lambda_0+p\lambda_1+\dots+p^{r-1}\lambda_{r-1}$ and $\mu=\mu_0+p\mu_1+\dots+p^{r-1}\mu_{r-1}$ be their $p$-adic expansions. Suppose we have $\lambda_i,\mu_i\in \bar C_\Z$ for each $i$. Then $\Ext^1_{G_r}(L(\lambda),L(\mu))=0$.\end{lemma}

\begin{proof}(i)
Clearly we may assume $s,t<r$. When $r=1$ the result follows from Lemma \ref{bnp}.
So assume $r>1$. 
Without loss of generality (dualising if necessary) we may assume $s\leq t$. Suppose $s>0$ and consider the following subsequence of the five-term exact sequence of the LHS spectral sequence applied to $G_s\triangleleft G_r$
(see \cite[I.6.10]{Jan03}):
\begin{align*}
0 \to
\Ext^1_{G_{r-s}}(L(\lambda)&,L(\mu)^{[t-s]}) \to
\Ext^1_{G_{r}}(L(\lambda)^{[s]},L(\mu)^{[t]})\\ &\to
\Hom_{G_{r-s}}(L(\lambda),\Ext^1_{G_s}(k,k)^{[-s]}\otimes L(\mu)^{[t-s]})
\to 0.
\end{align*}
Since $\Ext^1_{G_s}(k,k)=0$, we have
\[\Ext^1_{G_{r-s}}(L(\lambda),L(\mu)^{[t-s]})\cong \Ext^1_{G_{r}}(L(\lambda)^{[s]},L(\mu)^{[t]}),\]
and the left-hand side vanishes by induction, so we may assume $s=0$.
There is another exact sequence
\begin{align*}
0 \to 
\Ext^1_{G_{r-1}}(k, \Hom_{G_1}(L(\lambda)&,L(\mu)^{[t]})^{[-1]})
\to \Ext^1_{G_{r}}(L(\lambda),L(\mu)^{[t]})\\
&\to\Hom_{G_{r-1}}(k,\Ext^1_{G_1}(L(\lambda),L(\mu)^{[t]})^{[-1]})=0,
\end{align*}
where the last term vanishes by induction. 
If $t=0$ then as $\lambda\neq \mu$, 
the first term of the sequence vanishes and we are done. 
So we may assume $t>0$. Now we can rewrite the first term as
$\Ext^1_{G_{r-1}}(k, \Hom_{G_1}(L(\lambda),k)^{[-1]}\otimes L(\mu)^{[t-1]})$.
If this expression is non-trivial we have $\lambda=0$ and 
$\Ext^1_{G_{r-1}}(k,L(\mu)^{[t-1]})$ vanishes
by induction, which completes the proof.

(ii) Suppose $i$ is the first time either $\lambda_{i-1}$ or $\mu_{i-1}$ is non-zero. Without loss of generality, $\lambda_{i-1}\neq 0$. Write $\lambda=\lambda^i+p^i\lambda'$ and take a similar expression for $\mu$. Then there is an exact sequence
\begin{align*}
0 \to \Ext^1_{G_{r-i}}&(L(\lambda'),\Hom_{G_i}(L(\lambda^i),L(\mu^i))^{[-i]}\otimes L(\mu'))\to \Ext^1_{G_{r}}(L(\lambda),L(\mu))\\&\to\Hom_{G_{r-i}}(L(\lambda'),\Ext^1_{G_i}(L(\lambda^i),L(\mu^i))^{[-i]}\otimes L(\mu')).\end{align*} We have $L(\lambda^i)=L(\lambda_{i-1})^{[i-1]}$ and $L(\mu^i)=L(\mu_{i-1})^{[i-1]}$. Hence the right-hand term vanishes by part (i). The left-hand term is non-zero only if $\lambda^i=\mu^i$ and then we get 
$\Ext^1_{G_{r-i}}(L(\lambda'),L(\mu'))\cong \Ext^1_{G_{r}}(L(\lambda),L(\mu))$. Thus the result follows by induction on $r$.\end{proof}

We put these results together to arrive at an analogue of Jantzen's well-known result \cite{Jan97} that $G$-modules for which $\dim V\leq p$ are completely reducible. 

\begin{prop}\label{grcompred}Let $G$ be a simple, simply connected algebraic group and let $\dim V\leq p$ be a $G_r$-module. Then exactly one of the following holds:

(i) $V$ is a semisimple $G_r$-module;

(ii) $G$ is of type $A_1$, $p>2$, $r=1$, $\dim V=p$ and $V$ is uniserial, with composition factors $L(p-2-s)$ and $L(s)$ with $0\leq s\leq p-2$;

(iii) $G$ is of type $C_n$ with $n\geq 1$, $p=2$ and $V$ is uniserial with two trivial composition factors.\end{prop}
\begin{proof}
Assume $V$ has only trivial composition factors. We have $\Ext^1_{G_r}(k,k)\neq 0$ if and only if $p=2$ and $G$ is of type $C_n$, in which case $\Ext^1_{G_r}(k,k)^{[-r]}\cong L(\omega_1)$; \cite[II.12.2]{Jan03}. This is case (iii).

Otherwise, $p>2$ and $\Ext^1_{G_r}(L(\lambda),L(\lambda))=0$ for all $\lambda\in X_r(T)$ by \cite[II.12.9]{Jan03}. 

Assume $G$ is not of type $A_1$. By assumption, $V$ has a non-trivial composition factor with $\dim V\leq p$.
Then $p>2$ and the hypotheses of Lemma \ref{bnp} hold.
Since $\dim V\leq p$, by Proposition \ref{mcn1} any $G_r$-composition factor $L(\lambda)$ of $V$ has a $p$-adic expansion $\lambda=\lambda_0+\dots+p^{r-1}\lambda_r$ with each $\lambda_i\in \bar C_\Z$. If there were a non-split extension $0\to L(\lambda)\to V\to V/L(\lambda)\to 0$ then there would be a non-split extension of $L(\lambda)$ by $L(\mu)$ for $L(\mu)$ a composition factor of $V$, also of the same form. But by Lemma \ref{souped}(ii) we have $\Ext^1_{G_r}(L(\lambda),L(\mu))=0$, hence this is impossible and $L(\lambda)$ splits off as a direct summand. Induction on the direct complement completes the proof in this case.

If $G$ is of type $A_1$ then the $G_r$-extensions of simple modules are well known. If $r>1$ with $\lambda,\mu\in X_r(T)$ then $\dim\Ext^1_{G_r}(L(\lambda),L(\mu))=\dim\Ext^1_{G}(L(\lambda),L(\mu))$ and this must vanish whenever $\dim L(\lambda)+\dim L(\mu)\leq p$. If $r=1$, then the only pairs of $G_1$-linked weights are $s$ and $p-2-s$ with $\Ext^1_{G_1}(L(s),L(p-2-s))\cong L(1)^{[1]}$ as $G$-modules. Here we have $\dim L(s)+\dim L(p-s-2)=p$ giving case (ii).
\end{proof}

The following two corollaries are immediate, in the first case, the passage from $G$ being simple to being reductive is trivial
(consider the cover of $G$ by the product of the radical and the simply connected
cover of the derived group).
\begin{cor}\label{grsemisimple}
Let $G$ be a connected reductive algebraic group and let $V$ be a $G_r$-module with $p>\dim V$.
Then $V$ is semisimple.\end{cor}

\begin{cor}\label{corforred}
Let $G$ be connected reductive and $G_r\leq \GL(V)$ with $\dim V\leq p$. Then either $G_r$ is completely reducible on $V$ or $\dim V=p$, $G$ is of type $A_1$, $r=1$ and $G_r$ is in a maximal parabolic of $\GL(V)$ acting 
indecomposably on $V$ as described in case (ii) of Proposition \ref{grcompred}.

Moreover, if $\g$ is a $p$-reductive subalgebra of $\GL(V)$ with $\dim V < p$
then $\g$ acts semisimply on $V$.
\end{cor}

\begin{proof}If $G$ is not simple, it can be written as $HK$ with $H$ and $K$ non-trivial mutually centralising connected reductive subgroups with maximal tori $S$ and $T$ say. The Frobenius kernels $H_1,K_1\leq G_1\leq G_r$ are also mutually centralising, so that $H_1$ is in the centraliser of $T_1$. Now the centraliser of $T_1$ is a proper Levi subgroup of $\GL(V)$, hence restriction of $V$ to $H_r$ has at least one trivial direct factor, with direct complement $W$ say, $\dim W<p$. Thus by Corollary \ref{grsemisimple}, $W$ is completely reducible for $H_r$ and by symmetry, for $K_r$. Thus $W$ is completely reducible for $K_rH_r=G_r$.

Otherwise, $G$ is simple and Proposition \ref{grcompred} gives the result
(note that case (iii) does not occur due to dimension restrictions).

For the last part, Lemma \ref{predimpstrong} implies that $\g$
is the direct sum of a semisimple ideal and a torus,
and we may hence assume that $\g$ is a semisimple
restricted subalgebra of $\gl(V)$.
If $\g$ is not irreducible on $V$, then by Theorem \ref{Strade}  
there exists a semisimple group $G$ with $\g = \Lie(G)$.
Now the result follows from the case $G_1$ above.
\end{proof}

\subsection{Cohomology and complete reducibility 
for small \texorpdfstring{$\g$}{g}-modules}

We now transfer our results to the ordinary Lie algebra cohomology for $\g$.

Recall the exact sequence \cite[I.9.19(1)]{Jan03}:\begin{align}\label{seq}0&\to \opH^1(G_1,L)\to \opH^1(\g,L)\to\Hom^s(\g,L^\g)\nonumber\\
&\to \opH^2(G_1,L)\to \opH^2(\g,L)\to \Hom^s(\g,\opH^1(\g,L))\end{align}

The following theorem is the major result of this section.

\begin{theorem}\label{h1h2vanish} Let $\g=\Lie(G)$ be semisimple. Then: 
\begin{itemize}
\item[(a)] If $p>h$ with $\mu\in \bar C_\Z$ then either $\opH^2(\g,L(\mu))=0$, or one of the following holds: (i) $\g$ contains a factor $\sl_3$ and $L(\mu)$ contains a tensor factor of $L(p-3,0)$ or $L(0,p-3)$ for this $\sl_3$; (ii) $\g$ contains a factor $\sl_2$ and $L(\mu)$ has a tensor factor $L(p-2)$ for this $\sl_2$.
\item[(b)] If $p>2$ is very good for $G$ then $\opH^2(\g,k)=0$.
\item[(c)] If $p>2$ is very good for $G$ and $\lambda,\mu\in \bar C_\Z$ we have $\Ext^1_\g(L(\lambda),L(\mu))=0$, or $G$ contains a factor of type $A_1$, $L(\lambda)$ and $L(\mu)$ are simple modules for that factor,
$\lambda=s<p-1$, $\mu=p-2-s$ and we have $\Ext^1_\g(L(\lambda),L(\mu))^{[-1]}\cong L(1)$.
\end{itemize}
\end{theorem}
\begin{proof}
We may assume that $G$ is simply connected, since the condition on $p$ implies that $\g=\g_1\times \g_2\dots\times \g_s$. 
Now one can reduce to the case that $G$ is simple using a K\"unneth formula.
To begin with, any simple module $L(\lambda)$ for $\g=\g_1\times \g_2\times \dots \times \g_s$ is a tensor product of simple modules $L(\lambda_1)\otimes\dots\otimes L(\lambda_s)$ for the factors. Then by the K\"unneth formula $\dim \Ext^1_\g(L(\lambda),L(\mu))\neq 0$ implies that $\lambda_i=\mu_i$ for all $i\neq j$, some $1\leq j\leq s$ and $\Ext^1_\g(L(\lambda),L(\mu))\cong \Ext^1_{\g_j}(L(\lambda_j),L(\mu_j))$.
This means we may assume $G$ to be simple in (c).
For $\opH^2(\g,L(\lambda))$ to be non-zero one must have all $\lambda_i=0$ for all $i\neq j,k$ some $1\leq j<k\leq s$ and then
\begin{align*}
\opH^2(\g,L(\lambda))=
\opH^2(\g_j,L(\lambda_j))\otimes \opH^0(\g_k,L(\lambda_k))
&\oplus \opH^1(\g_j,L(\lambda_j))\otimes \opH^1(\g_k,L(\lambda_k))\\
&\oplus \opH^0(\g_j,L(\lambda_j))\otimes \opH^2(\g_k,L(\lambda_k)).
\end{align*}

Now first suppose that both $\lambda_j$ and $\lambda_k$ are non-trivial.
Then only the second direct summand in $\opH^2(\g,L(\lambda))$ survives,
and by (\ref{seq}) it coincides with the tensor product of the $1$-cohomology
groups of the corresponding Frobenius kernels.
By Proposition \ref{mcncor}, non-vanishing would force $\lambda_j = p-2 = \lambda_k$ and
$\g_j = \g_k = \sl_2$ giving one exceptional case.

Next we treat the case $\lambda_k = 0$ and $\lambda_j$ non-trivial.
Again by (\ref{seq}) and Proposition \ref{mcncor}, we obtain
$\opH^2(\g,L(\lambda)) = \opH^2(\g_j,L(\lambda_j))$,
and we are in the case where $G$ is simple and $L(\lambda)$ non-trivial.
In case $\g=\sl_2$, the
result follows from \cite{Dzhuma}. So suppose $\g \neq \sl_2$.
Setting $L=L(\mu)$ in (\ref{seq}) we see that if $\mu\neq 0$ we have $\opH^1(\g,L)\cong \opH^1(G_1,L)$ and the right-hand side is zero by Lemma \ref{bnp}. Thus we also have $\opH^2(\g,L)\cong \opH^2(G_1,L)$ and the latter is zero by Proposition \ref{mcncor} unless $\g=\sl_3$ and the exception is as in the statement of the Theorem,
since we have excluded the $A_1$ case.

Finally, the case $\lambda_j = \lambda_k = 0$ reduces by the above
to the case $G$ simple, $L=k$ and the claim that $\opH^2(\g,k)=0$.
Here we have $\opH^1(\g,k)\cong (\g/[\g,\g])^*$ and this is zero since $p$ is very good and $\g$ is semisimple. We also have $\opH^2(G_1,k)^{[-1]}\cong \g^*$.
The injective map $\Hom^s(\g,L^\g)\to \opH^2(G_1,L)$ is hence an isomorphism,
which forces $\opH^2(\g,k)=0$ in the sequence (\ref{seq}).
This also proves (b).

Now we prove the statement (c) under the assumption that $G$ is simple.
We have an isomorphism $\Ext^1_\g(L(\lambda),L(\mu))\cong \opH^1(\g,L(\mu)\otimes L(\lambda)^*)$. Let $M=L(\mu)\otimes L(\lambda)^*$. If $\lambda\neq \mu$, then applying the exact sequence (\ref{seq}) to $M$ yields $\opH^1(\g,M)\cong \opH^1(G_1,M)$ and the latter is zero by Lemma \ref{bnp} if $G$ is not of type $A_1$ and well-known if $G$ is of type $A_1$. Hence we may assume $\lambda=\mu$. The assignation of $L$ to the sequence (\ref{seq}) is functorial, thus, associated to the $G$-map $k\to M\cong \Hom_k(L,L)$, there is a commutative diagram
\[\begin{CD}0@>>> \opH^1(\g,k)=0 @>>> \Hom^s(\g,k^\g)\cong(\g^*)^{[1]} @>\cong>> \opH^2(G_1,k)\\
@. @VVV @V\cong VV @V\theta VV \\
0 @>>> \opH^1(\g,M) @>>> \Hom^s(\g,M^\g)\cong(\g^*)^{[1]} @>\zeta >> \opH^2(G_1,M)\end{CD},\] where the natural isomorphism $k^\g\to M^\g$ induces the middle isomorphism and the top right isomorphism has been discussed already. We want to show that $\zeta$ is injective, since then it would follow that $\opH^1(\g,M)=0$. To do this it suffices to show that $\theta$ is an injection $(\g^*)^{[1]}\to \opH^2(G_1,M)$ and for this, it suffices to show that the simple $G$-module $(\g^*)^{[1]}$ does not appear as a submodule of $\opH^1(G_1,M/k)$. Now since $\lambda\in \bar C_\Z$ we have $L(\lambda)\cong H^0(\lambda)$ and so by \cite[II.4.21]{Jan03}, $M$ has a good filtration. The socle of any module $H^0(\mu)$ with $\mu\in X^+$ is simple. Thus the submodule $k\leq M$ constitutes a section of this good filtration, with $M/k$ also having a good filtration.

The $G$-modules $\opH^1(G_1,H^0(\mu))$ have been well-studied by Jantzen \cite{Jan91} and others. In order to have $(\g^*)^{[1]}$ a composition factor of $\opH^1(G_1,H^0(\mu))$, we would need $\g\cong\g^*\cong H^0(\omega_\alpha)$ where $\mu=p\omega_\alpha-\alpha$ and $\alpha$ is a simple root with $\omega$ the corresponding fundamental dominant weight; \cite[Theorem 3.1(A,B)]{BNP04-Frob}. Now for type $A_n$, with $p\not|n+1$, we have $\g=L(2\omega_1)$ if $n=1$ and $\g=L(\omega_1+\omega_n)$ else; and for type $B_2$, we have $\g=L(2\omega_2)$, ruling these cases out. For the remaining types, we have
\begin{center}\begin{tabular}{l|lllllll}\hline
Type & $B_n$,$C_n$ & $D_n$ & $E_6$ & $E_7$ & $E_8$ & $F_4$ & $G_2$ \\\hline
$\g\cong L(\omega_\alpha)$ for $\omega_\alpha=$ & $\omega_2$ & $\omega_2$ & $\omega_2$ & $\omega_1$ & $\omega_8$ & $\omega_1$ & $\omega_2$\\
$\langle p\omega_\alpha-\alpha,\alpha_0^\vee\rangle$ & $2p$ & $2p$ & $2p-1$ & $2p-1$ & $2p-1$& $2p$ & $3p$\end{tabular}\end{center}On the other hand, since $\lambda\in \bar C_\Z$ it satisfies $\langle\lambda+\rho,\alpha_0^\vee\rangle\leq p$, i.e. $\langle\lambda,\alpha_0^\vee\rangle\leq p-h+1$. Hence any high weight $\mu$ of $M=L\otimes L^*$ satisfies $\langle\mu,\alpha_0^\vee\rangle\leq 2p-2h+2$. Looking at the above table, it is easily seen that this is a contradiction. Thus $(\g^*)^{[1]}$ is not a composition factor of $\opH^1(G_1,M/k)$ and the result follows.
 \end{proof}

\begin{remarks}(i) When $\lambda\neq \mu$ in the proof of the above proposition, one also sees that there is an isomorphism $\Ext^2_{G_1}(L(\lambda),L(\mu))\cong \Ext^2_{\g}(L(\lambda),L(\mu))$ but we do not use this fact in the sequel.

(ii) The conclusion of the theorem is incorrect if $G$ is reductive but not semisimple. For example, if $G$ is a torus, then $\g$ is an abelian Lie algebra, and $\opH^1(\g,k)$ is non-trivial. For instance the two-dimensional non-abelian Lie algebra is a non-direct extension of $k$ by $k$. One also has $\opH^2(k\times k,k)\neq 0$ by the K\"unneth formula: for example the Heisenberg Lie algebra is a non-split extension of $k$ by $k\times k$.

(iii) When $p=3$ and $G=\SL_3$, then $H^2(G_1,k)^{[-1]}\cong \g^*\oplus L(\omega_1)\oplus L(\omega_2)$, by \cite[Theorem 6.2]{BNP07}. Thus the same argument shows that $\opH^2(\g,k)\cong L(\omega_1)\oplus L(\omega_2)$. It follows from the K\"unneth formula that if $G$ is a direct product of $n$ copies of $\SL_3$ then $\opH^2(\g,k)\cong [L(\omega_1)\oplus L(\omega_2)]^{\oplus n}$.

(iv) In part (a) of the theorem, one can be more specific. If $\g=\sl_2$ then \cite{Dzhuma} shows that $\opH^2(\g,L(p-2))$ is isomorphic to $L(1)^{[1]}$ as a $G$-module. If $\g=\underbrace{\sl_2\times\dots\times \sl_2}_{n\text{ times}}\times \h$ then one can show moreover that $\opH^2(\g,L(\mu))$ is non-zero only if \[L(\mu)\cong L(\mu_1)\otimes \dots \otimes L(\mu_n)\otimes L(\mu_{n+1})\] with each $\mu_i\in \{0,p-2\}$ and $\mu_{n+1}=0$. Let $r$ be the number of times $\mu_i=p-2$. Then, the K\"unneth formula shows that \[\dim\opH^2(\g,L(\mu))=\begin{cases}0 \text{ if }r=0;\\
2 \text{ if }r=1;\\
4 \text{ if }r=2;\\
0 \text{ otherwise}.\end{cases}\]
\end{remarks}

We use the theorem above to get analogues of Corollary \ref{corforred} for Lie algebra representations.

\begin{prop}\label{liecompred} Let $G$ be a simple algebraic group with $\g=[\g,\g]$ and let $\dim V\leq p$ be a $\g$-module. Then exactly one of the following holds:

(i) $V$ is a semisimple $\g$-module;

(ii) $G$ is of type $A_1$, $\dim V=p$ and $V$ is uniserial, with composition factors $L(p-2-s)$ and $L(s)$.\end{prop}
\begin{proof}The proof is similar to Proposition \ref{grcompred}. Since $\dim V\leq p$, any composition factor of $V$ is a restricted simple $\g$-module, or $V$ is simple. Since $\Ext^1_\g(k,k)=\opH^1(\g,k)\cong (\g/[\g,\g])^*=0$, if $V$ consists only of trivial composition factors then $V$ is semisimple. Thus we may assume that $\g$ contains a non-trivial composition factor $L$. Then either $\dim L=p$ and $V$ is simple, or $p>h$ by Lemma \ref{mcn1}(iii). By the condition on $V$, any two distinct composition factors, $L(\lambda)$ and $L(\mu)$ satisfy $\lambda,\mu\in C_\Z$ by Lemma \ref{mcn1}(ii). If $G$ is not of type $A_1$, then $\Ext^1_\g(L(\lambda),L(\mu))=0$ by Theorem \ref{h1h2vanish} and the exceptional case, where $G=A_1$, is well known.\end{proof}

As before there is a corollary:

\begin{cor}\label{liealgcompred}Let $G$ be a semisimple algebraic group and let $V$ be a $\g$-module with $p>\dim V$.
Assume that $\g = [\g,\g]$. Then $V$ is semisimple.\end{cor}

The next corollary uses a famous result of Serre on the semisimplicity of tensor products to extend our results a little further. This result will be crucial for showing the splitting of certain non-semisimple Lie algebras.

\begin{cor}\label{SerCor} Let $\g$ be a Lie algebra and $V$, $W$ two semisimple $\g$-modules with $\dim V+\dim W<p+2$. Then $V\otimes W$ is semisimple. 

Furthermore, let $\g=\Lie(G)$ for $G$ a semisimple algebraic group with $p>2$ and $p$ very good. Then $\opH^2(\g,V\otimes W)=0$ unless $\g$ contains a factor $\sl_2$ and $V\otimes W$ contains a composition factor of the $\sl_2$-module $L(p-2)$. Also $\opH^1(\g,V\otimes W)=0$, unless one of $V$ and $W$ is isomorphic to $k$ and we are in one of the exceptional case of Theorem \ref{h1h2vanish}.\end{cor}

\begin{proof} For the first statement, we begin with some
reductions as in \cite{Ser94s}.
If $W=0$ or $k$ there is nothing to prove. If $W$ is at least $2$-dimensional, then either $p=2$ and $V$ is trivial (so that the result holds), or both $\dim V$ and $\dim W<p$. We may assume that both $V$ and $W$ are simple. Further, we may replace $\g$ by the restricted algebra generated by its image in $\gl(V\oplus W)$.
As $V \oplus W$ is a semisimple module, we may thus assume $\g$ is $p$-reductive. 
Now $\g \subseteq \gl(V) \times \gl(W) = \sl(V) \times \sl(W) \times \z$, where $\z$ is a torus,
and where the projections of $\g$ onto the first two factors are irreducible, hence
semisimple by Theorem B.
We thus may assume $\g \subseteq \sl(V) \times \sl(W)$ is a semisimple restricted subalgebra.

By Theorem \ref{Strade}, either (i) $\g$ has a factor $W_1$, the first Witt algebra and $V$ is the $(p-1)$-dimensional irreducible module for $W_1$; or (ii) $\g$ is $\Lie(G)$ for a direct product of simple algebraic groups, and $V$ and $W$ are (the differentials of) $p$-restricted modules for $G$. In case (i), as $p>2$, we would have $W\cong k\oplus k$ for $W_1$ and the result holds. So we may assume that (ii) holds. Now \cite[Prop.~7]{Ser94s} implies that $V\otimes W$ is the direct sum of simple modules with restricted high weights $\lambda$ satisfying $\lambda\in C_\Z$. Since each of these composition factors is simple also for $\g$, $V\otimes W$ is semisimple with those same composition factors.

For the remaining statements,
let $\h$ be the image of $\g$ in $\gl(V\oplus W)$,
so that $\g=\h\oplus\s$ with $\s$ acting trivially.
Let $h$ be the coxeter number of $\h$.
Now
if $W=k$, say, then since $p$ is very good for $\g$ we can have $p=\dim V$ by Proposition \ref{mcn1}
only for $p>h$, so otherwise $\dim V<p$.
And if $\dim W>1$ then $\dim V<p$ also. Now $\dim V<p$ also implies by Proposition \ref{mcn1} that $p>h$.
Also a summand $L(\lambda)$ of $V\otimes W$ has $\lambda\in C_\Z$. Now Theorem \ref{h1h2vanish} implies that $\opH^1(\g,V\otimes W)=\opH^2(\g,V\otimes W)=0$, unless we are in the exceptional cases described. However, if $\g=\sl_3$ then the module $L(p-3,0)$ or its dual has dimension $(p-1)(p-2)(p-3)/2>((p+1)/2)^2$ hence it cannot appear as a composition factor of $V\otimes W$.
\end{proof}

\begin{remark} If $\g=W_1$ the conclusion of the second part is false, since $H^1(\g,V)\neq 0$ when $V$ is the irreducible $(p-1)$-dimensional module for $\g$.\end{remark}

\begin{proof}[Proof of Theorem D:] We must just give references for the statements made. For (a), see Proposition \ref{grcompred}; for (b), see Proposition \ref{liecompred}; for (c), see Theorem \ref{h1h2vanish}; for (d), see Corollary \ref{SerCor}. This completes the proof of Theorem D.\end{proof}

\section{Decomposability: the existence of Levi factors}

Let $\h$ be a restricted subalgebra of $\gl(V)$ with $p>\dim V$. In this section we show, in Theorem \ref{splittingLie}, a strong version of the Borel--Tits Theorem in this context. 

Let $G$ be connected reductive. Recall, say from \cite{ABS90} that if $\p=\l+\q$ is a parabolic subalgebra of $\g=\Lie G$ then $\q$ has a central filtration such that successive quotients have the structure of $\l$-modules. We record a specific case:

\begin{lemma}\label{abslem}In case $G=\GL_n$, a parabolic subalgebra $\p=\l+\q$ has the property that $\l$ is a direct product $\gl(V_1)\times \gl(V_2)\times\dots\times \gl(V_r)$ and $\q$ has a central filtration with successive factors being modules of the form $V_i\otimes V_j^*$, each factor occurring exactly once.\end{lemma}

\begin{theorem}\label{splittingLie}Let $\h$ be a restricted Lie subalgebra of $\gl(V)$ with $\dim V<p$, and let $\r=\Rad_p(\h)$ ($=\Rad_V(\h)$). 

Then there is a parabolic subalgebra $\p=\l+\q$, with $\r\leq \q$ and containing a complement $\s$ to $\r$ in $\h$, with $\s\leq \l$ and $\h=\s+\r$ as a semidirect product.  Furthermore, $\s$ acts completely reducibly on $V$ and is 
the direct sum of a torus and a semisimple ideal.\end{theorem}

\begin{proof} As in the proof of Lemma \ref{predimpstrong}  we take a minimal parabolic subgroup $P=LQ$ so that its Lie algebra $\p=\l+\q$ contains $\h$ and so that the projection $\h_\l:=\pi(\h)$ of $\h$ to the Levi subalgebra $\l$ is strongly $p$-reductive and we may write $\h_\l=\h_s\oplus\z$ where $\h_s$ is semisimple and $\z=Z(\h_\l)$.
We also have $\r \leq \q$, since $\h_\l$ is $p$-reductive.

Now by Theorem \ref{Strade}, either $\h_s=W_1$,
$\h = \h_\l$,
$\p=\l=\gl(V)$ and we are done;
or $\h_\l$ is isomorphic to a direct product of classical Lie algebras $\s_i$ and $\z$.  

We first lift $\z$ to $\h$.
Let $\pi':\h \rightarrow \z$ be the composition of $\pi$ with the projection onto $\z$.
By \cite[Lemma 2.4.4(2)]{SF88}, there is a torus $\z'\leq Z(\l)+\q$ so that $\h=\z'+\ker(\pi')$. Now since $\z'$ is a torus, it is linearly reductive, we may replace $\h$ by a conjugate by $Q$ so that $\z'\subseteq Z(\l)$. Let us rewrite $\z=\z'$ and identify $\z$ with its image in $\l$ under $\pi$. 

Next we construct a complement to $\r$ in $\h$.
Let $\pi'':\h\to\h_s$ be the composition of $\pi$ with the projection onto $\h_s$ and let $\h' \subseteq \h$ be 
a vector space complement to $\ker(\pi'')$.
Then $\r + \h' \leq \h$ is a subalgebra, and we have an exact sequence
\[ 0 \to \r \to \r + \h' \stackrel{\pi''}{\to} \h_s\to 0.\]

We show this sequence is split. By Lemma \ref{abslem}, the nilpotent radical $\q$ of $\l$ has a filtration $\q=\q_1\supseteq \q_2\supseteq\dots \supseteq\q_m=0$ with each $\q_i/\q_{i+1}$ having the structure of an $\l$-module $M_i\otimes N_i$ with $M_i$ and $N_i$ irreducible modules for the projections of $\h_\l$ to distinct factors of the Levi. Since $\dim M_i+\dim N_i<p$, we have by Corollary \ref{SerCor} that $M_i\otimes N_i$ is a direct sum of irreducible modules for $\h_s$ with $\opH^2(\h_s,M_i\otimes N_i)=0$. By intersecting with $\r$, we get a filtration $\r=\r_1\supseteq \r_2\supseteq\dots \supseteq\r_m=0$ by $\h_s$-modules so that each $\r_i/\r_{i+1}$ is a submodule of $M_i\otimes N_i$, hence also a semisimple module with $\opH^2(\h_s,\r_i/\r_{i+1})=0$. By an obvious induction on the length $m$ of the filtration $\{\r_i\}$ we now see that the sequence \[0\to\r\to\r+\h'\to\h_s\to 0\] is split. Thus we may set $\h_s'$ a complement to $\r$ in $\h'+\r$.

We would like to set $\s =\h_s'+\z$, however this vector space may not be a subalgebra of $\g$.
Write $\q=\c_\q(\z)+[\q,\z]$. (This can be done, for instance by \cite[Lemma 2.4.4(1)]{SF88}.)
Any element $h$ of $\h_s'$ can be written as $h_1+q_1+q_2$
for $h_1\in\l$, $q_1$ in $\c_\q(\z)$ and $q_2\in[\q,\z]$.
As $\h$ is stable under $\ad\z$, with $\z$ centralising $h_1$ and $q_1$, we conclude that $q_2\in\h$. Thus we have the element $h'=h_1+q_1\in\h$.
Thus we may form the subspace $\h_s'' \leq \h$ with $\h_s''\leq\l+\c_\q(\z)$.

Using that $\h_s' \leq \h$ is a subalgebra, that $\c_\q(\z)$ is
$\l=\c_{\gl(V)}(\z)$-invariant and that $[\q,\z]$ is an ideal in $\q$, one checks that $\h_s''$ is indeed a subalgebra,\footnote{The calculation is as follows: if $h_1+q_1+q_2$ and $h_1'+q_1'+q_2'$ are two elements of $\h_s'$ then 
\[[h_1+q_1+q_2, h_1'+q_1'+q_2']=\underbrace{[h_1,h_2]}_{\in\h_\l}+\underbrace{[h_1,q_1']+[q_1,h_1']+[q_1,q_1']}_{\in\c_\q(\z)} + x,\] where $x\in [\q,\z]$ by the Jacobi identity. But projecting to $\h_s''$ one simply deletes $q_2$, $q_2'$ and $x$ to get the analagous calculation.} with $\h_s''$ also a complement to $\r$ in $\h_\s'+\r$. Now we have guaranteed that $\s=\h_s''+\z$ is a subalgebra of $\h$,
a complement to $\r$ in $\h$.

Now, by Corollary \ref{liealgcompred}, $\h_s''$ acts completely reducibly. Also, since $\z$ is a torus, $\z$ is linearly reductive on restricted representations, hence also acts completely reducibly. Thus $\s$ is completely reducible on $V$.
In particular, we may replace $\l$ with a Levi subalgebra of $\p$ that contains $\s$,
which finishes the proof.
\end{proof}

\section{Proof of Theorem B(i)}

\begin{proof}We first prove the statement in the case that $G=\GL(V)$, so we assume $p>\dim V+1$.
By assumption, $\h$ is a restricted subalgebra of $\g$.

Let $\n=\n_\g(\h)$. By Theorem \ref{splittingLie} we may decompose both $\n$ and $\h$.
Let $\n=\n_\l+\n_\q\leq \p=\l+\q$ with $\n_\l\leq \l$ and $\n_\q\leq\q$, with $\n_\l=\n_s+\z$, $\z$ a torus and $\n_s$ is by Theorem \ref{Strade} isomorphic to a direct product of classical Lie algebras acting completely reducibly on $V$; also set $\h_\q=\h\cap\q$ and $\h_\l=\pi(\h)$ the projection to $\l$. Since $\n$ is generated by nilpotent elements we have $\z=0$ and $\h_\l=\h_s$.
Since the complement to $\h_\q$ in $\h$ obtained by Theorem \ref{splittingLie}
is completely reducible on $V$ and hence conjugate to a subalgebra of $\l$, we may
assume that $\h = \h_\q + \h_\l$ is this splitting. Furthermore, $\h_\l\leq\n_\l$ is an ideal of a direct product of simple subalgebras, hence is a direct product of some subset of those simples. 

Since $V$ has dimension less than $p$, $V|_{\n_\l}$ is a restricted module for $\n_\l$. Hence there is a connected algebraic group $N_\l$ with $\Lie N_\l\cong \n_\l$, $N_\l\leq \GL(V)$ and $V|_{\Lie(N_\l)}\cong V|_{\n_\l}$. Hence, replacing $N_\l$ by a conjugate if necessary, we have $\Lie(N_\l)=\n_\l$. Moreover if $L$ is a Levi subgroup of $\GL(V)$ chosen so that $\Lie(L)=\l$ then 
we may produce $N_\l\leq L$. Clearly $N_\l$ normalises any direct factor of $\n_\l$, in particular, $\h_\l$. 

Now, since the $\l$-composition factors of $\q$ are all of the form $W_1\otimes W_2$ for $\dim W_1+\dim W_2<p$ and $W_1$, $W_2$ irreducible for $\n_\s$, \cite[Prop.\ 7]{Ser94s} implies that $\q$ is a restricted semisimple module for $N_\l$ and $\n_\l$. Since $\n_\l$ normalises $\h_\q=\h\cap \q$, this space also appears as an $N_\l$-submodule
in $\q$, hence $N_\l$ normalises $\h_\q$.

It remains to construct a unipotent algebraic group $N_\q$ such that $\Lie N_\q=\n_\q$ with $N_\q$ normalising $\h$. For this we use Corollary \ref{expParab}. Let $N_\q=\overline{\langle \exp x:x\in\n_\q\rangle}$. Then $N_\q$ is a closed subgroup, which by Corollary \ref{expParab}  consists of elements normalising $\h$.
By Lemma \ref{genbynilp}, $\n_\q\leq\Lie(N_\q)$.

Let $N$ be the smooth algebraic group given by $N=\langle N_\l,N_\q\rangle$. We have shown that $N$ normalises $\h$ and that $\n\subseteq \Lie N$. Since also $\Lie N\subseteq \n$ we are done for the case $G=\GL(V)$.

To prove the remaining part, we appeal to Proposition E again. 

Let $G$ be a simple algebraic group with minimal dimensional representation $V$. Then since $p>\dim V$, $(\GL(V),G)$ is a reductive pair.
Indeed, the assumption on $p$ guarantees that the trace form
associated to $V$ is non-zero, see \cite[Fact 4.4]{Garibaldi}.
This implies the reductive pair property (cf.\ the proof
of \cite[Prop.\ 8.1]{Garibaldi}).
The theorem now follows by invoking Proposition E.
\end{proof}

\section{Examples}\label{sec:examples}

In this section we mainly collect, in a number of statements, examples which demonstrate the tightness of some of our bounds. First let us just point out that there are some rather general situations in which smooth normalisers can be found.

\begin{example}[\!\!{\cite[Theorem B]{McT09}}]Suppose $G$ is a quasi-split reductive group over a field $k$ of very good characteristic. Then the normaliser $N=N_G(C)$ of the centraliser $C=C_G(e)$ of a regular nilpotent element $e$ of $\g=\Lie(G)$ is smooth.\end{example}

\begin{example}[\!\!{\cite[Proof of Lem.~3.1]{HSMax}}]Suppose $G$ is reductive over an algebraically closed field $k$ of very good characteristic and $e$ is a nilpotent element of $\g=\Lie(G)$, then the normaliser $N_G(\la e\ra)$ of the $1$-space $\la e\ra$ of $\g$ is smooth.\end{example}

We will first give the promised example discussed after the statement of Theorem A. For this, we will need a lemma.

%\begin{lemma}\label{NBhSmoothLift} Suppose $N_B(\h)$ is smooth and $s\in\b$ is a diagonal element normalising a subspace $\h$ of $\u$. Then $\la s\ra$ lifts to the image of a cocharacter $\chi$ of $N_B(\h)$ which is conjugate by an element of $C_{U}(s)$ to a cocharacter with diagonal image.\end{lemma}
%\begin{proof} Since $N_B(\h)$ is smooth, we may, 
%by \cite[Thm.\ 13.3]{Hum67}, write any maximal torus $\t$ of $\n_B(\h)$ as $\Lie(T)$ for $T$ a maximal torus of $N_B(\h)$. By \cite[Prop.\ 2]{Dieu:1953}, for any semisimple element $s\in T$ we may write $\la s\ra=\Lie(S)$ for $S\subseteq T$. We may even write $s=d/dt|_{t=0} \chi(t)$ for $\chi$ a cocharacter of $N_B(\h)$.
%
%Assume $s\in D$ for $D$ the diagonal torus. We have $S$ is conjugate to a diagonal torus via $u\in U$, say, i.e.\ $u\chi(t)u^{-1}$ is diagonal. Differentiating this, we have $d/dt|_{t=0}(u\chi(t)u^{-1})=s$, so that $us=su$, i.e.\ $u\in C_U(s)$.\end{proof}

\begin{lemma}\label{NBhSmoothLift} Let $B=TU$ be a Borel subgroup of a reductive algebraic group $G$ containing a maximal torus $T$ with unipotent radical $U$. Suppose $N_B(\h)$ is smooth and $s\in\t=\Lie(T)$ an element normalising a subspace $\h$ of $\u=\Lie(U)$. Then $\la s\ra=\Lie(\chi(\Gm))$ for a cocharacter $\chi:\Gm\to N_B(\h)$, such that $\chi(\Gm)$ is conjugate by an element of $C_{U}(s)$ to a cocharacter with image in $T$.\end{lemma}
\begin{proof} Since $N_B(\h)$ is smooth, we may, 
by \cite[Thm.\ 13.3]{Hum67}, write any maximal torus $\s$ of $\n_\b(\h)$ as $\Lie(S)$ for $S$ a maximal torus of $N_B(\h)$. By \cite[Prop.\ 2]{Dieu:1953}, for any semisimple element $s\in \s$ we may write $\la s\ra=\Lie(S_1)$ for $S_1\subseteq S$. Defining an appropriate isomorphism $\Gm\to S_1$, we may even write $s=\frac{\mathrm{d}}{\mathrm{d}t}\big|_{t=1} \chi(t)$ for $\chi$ a cocharacter of $N_B(\h)$.

As the maximal tori of $B$ are conjugate by elements of $U$, we have that $S$ is conjugate to its projection to $T$, say via $u\in U$; in particular, $u\chi(t)u^{-1}\in T$. Since projection to $T$ is $B$-equivariant, we have on differentiating, that $\frac{\mathrm{d}}{\mathrm{d}t}\big|_{t=1}(u\chi(t)u^{-1})=s$, so that $usu^{-1}=s$, i.e.\ that $u\in C_U(s)$.\end{proof}

\begin{example}\label{fibonaccimotherfucker}
Let $n\geq 4$. This example depends on three fixed parameters $\lambda_1,\lambda_2,\lambda_3$ together with variables $\{a_i\}_{1\leq i\leq n}$, $\{b_i\}_{1\leq i\leq n-1}$, $c$, $d$, and $e$, each taking values in $k=\bar\F_p$.

Let us define the following matrices:

%% {\footnotesize
%% \[A:=\left(\begin{array}{cccccccccc}
%% 0 & a_1 & a_2 &*&*&*&\dots&*&*&*\\
%% &0&a_1&a_3+b_1&*&*&\dots&*&*&* \\
%% &&0&a_2&a_4+b_2&*&\dots&*&*&* \\
%% &&&0&a_3&a_5+b_3&\ddots&\vdots&\vdots&\vdots\\
%% &&&&0 &a_4&\ddots&*&*&*\\
%% &&&&&0 &\ddots&a_{n-1}+b_{n-3}& c &e+\lambda_1 (a_{n-1}+b_{n-3})\\
%% &&&&&&\ddots &a_{n-2}&a_n+b_{n-2}&(1+\lambda_1) a_{n-2}+d\\
%% &&&&&&&0&a_{n-1}&b_{n-1}\\
%% &&&&&&&&0&a_n\\
%% &&&&&&&&&0\end{array}\right),\]}

\[A:=\left(\begin{array}{cccccccccc}
0 & a_1 & a_2 &*&*&*&\dots&*&*&*\\
&0&a_1&\beta_2&*&*&\dots&*&*&* \\
&&0&a_2&\beta_3&*&\dots&*&*&* \\
&&&0&a_3&\beta_4&\ddots&\vdots&\vdots&\vdots\\
&&&&0 &a_4&\ddots&*&*&*\\
&&&&&0 &\ddots&\beta_{n-2}& c &e+\lambda_1 \beta_{n-2}\\
&&&&&&\ddots &a_{n-2}&\beta_{n-1}&(1+\lambda_1) a_{n-2}+d\\
&&&&&&&0&a_{n-1}&b_{n-1}\\
&&&&&&&&0&a_n\\
&&&&&&&&&0\end{array}\right),\]

with $\beta_i = a_{i+1} + b_{i-1}$ for $i=2,\dots,n-1$,

\begin{align*}B&:=\left(\begin{array}{cccccccccc}
0&a_1&b_1&*&\dots&*&\\
&0&a_2&b_2&\ddots&\vdots\\
&&0&a_3 & \ddots&*\\
&&&0&\ddots&b_{n-1}\\
&&&&\ddots&a_n&\\
&&&&&0\\
\end{array}\right),\\
\\
C&:=\left(\begin{array}{ccccccccc}
0 &a_{n-3}&a_{n-1}+b_{n-3}& c &e+\lambda_2 (a_{n-1}+b_{n-3})\\
&0 &a_{n-2}&a_n+b_{n-2}&(1+\lambda_2)a_{n-2}+d\\
&&0&a_{n-1}&b_{n-1}\\
&&&0&a_n\\
&&&&0\end{array}\right),\\
\\
D&:= \left(\begin{array}{cccccc}
0 &a_{n-2}&a_{n}+b_{n-2}& d + \lambda_3 a_{n-2}\\
&0 &a_{n-1}&b_{n-1}\\
&&0&a_{n}\\
&&&0\end{array}\right)\end{align*}

Then the reader may check that for each choice of $\lambda_1$, $\lambda_2$ and $\lambda_3$, the following set defines a subalgebra $\h$ of the strictly upper triangular matrices:%\footnote{By way of explanation for the production of these matrices apparently out of thin air, we explain our method. As we do below, one can check that diagonal matrices $M$ which normalise matrices of the form $A$, with all variables labelled $b_i$, $c$, $d$ and $e$ set equal to zero forces the entries of $M$ to satisfy Fibonacci-type relations. The remaining entries and matrices $B$, $C$ and $D$ serve simply to bolster the $A$-type matrices into a subalgebra in such a way that one does not `disturb' the relations on the entries of $M$ already built up. If one did not have something like the $b_i$ variables, then by taking commutators of $A$-type matrices with only $a_i$ entries one could generate almost all elements of the second super-diagonal. This would destroy the relationship of the second super-diagonal with the first super-diagonal.}:
\[\left\{\left(\begin{array}{cccc}A&*&*&*\\
0&B&*&*\\
0&0&C&*\\
0&0&0&D\end{array}\right):a_i\in k, b_j\in k, c,d,e\in k\right\}.\]

Let $F_i$ denote the $i$th Fibonacci number, so that $F_0=F_1=1$ and $F_2=2$ and suppose that $r$ is chosen so that $F_{r+1}=p$ is the prime characteristic of $k$, and let us suppose that $N_G(\h)$ is smooth. Since every entry of the superdiagonal is non-zero for some element in $\h$, it is easy to check that $N_G(\h)\subseteq B$. Thus $N_G(\h)=N_B(\h)$ and we may employ Lemma \ref{NBhSmoothLift}.

Suppose $s=\diag(s_1,\dots,s_{2n+12})$ is an arbitrary element of the diagonal torus $\t=\Lie(T)$. Then one can calculate the dimension of $\n_\t(\h)$ by enumerating the linear conditions amongst the $t_i$ necessary to normalise $\h$. For example, setting all indeterminates in a general matrix of $\h$ to be zero, except for $a_1=1$ gives a matrix $M$, which spans a $1$-space $\la M\ra$ of $\h$. One can see by inspection that $s$ will normalise $\h$ only if it normalises $\la M\ra$. However, calculating $[s,M]$, we see that to normalise $\la M\ra$ implies the following condition must hold: \[s_1-s_2=s_2-s_3=s_{2n+4}-s_{2n+5}.\] Repeating over other $1$-spaces leads to a collection of relations which can be expressed by a system of linear equations $R\mathbf s=0$ for some matrix $R$ and the vector $\mathbf s=(s_1,\dots,s_{2n+12})$. The kernel of $R$ modulo $p$ then determines the dimension of $\n_\t(\h)$. To determine the dimension of $N_T(\h)$, one searches for cocharacters $\chi(t)=\diag(t^{k_1},t^{k_2},\dots,t^{k_{2n+12}})$ 
which normalise $\h$ by conjugation. This leads to an identical set of relations on the entries of the vector $\mathbf k=(k_1,\dots,k_{2n+12})$, so that the equation $R\mathbf k=0$ must be solved \emph{over the integers}. Then the dimension of $N_T(\h)$ is the nullity of $R$.

The nullities of $R$ over $\Z$ and over $\Z/p$ are identical if and only if $s$ can be lifted to a diagonal cocharacter $\chi(t)$ so that $d/dt|_{t=1}\chi(t)=s$. By explicit calculation of $R$ in our particular case, one sees its elementary divisors are $0^4,1^{2n+7},F_{r+1}$. Thus since $p=F_{r+1}$ the nullity of $R$ modulo $p$ is bigger than over $\Z$. Thus there is a toral element $s$, which cannot be lifted to a diagonal cocharacter. In our case, $\h$ has an obvious centraliser whose elements are: \[\diag(\underbrace{s_1,\dots,s_1}_{r+2},\underbrace{s_2,\dots,s_2}_{r+1},\underbrace{s_3,\dots,s_3}_{5},\underbrace{s_4,\dots,s_4}_{4})\]
which accounts also for the four-dimensional kernel over the integers. 

One also checks that the subalgebra $\h$ is normalised by the toral element
\begin{align*}
s:=&\diag(1,2,3,5,8,\dots,F_r,F_{r+1},F_{r+2})\\
&\oplus\diag(F_2+4=6,F_3+4=7,\dots,F_r+4,F_{r+1}+4,F_{r+2}+4=F_r+4)\\
&\oplus\diag(F_{r-2}+1,F_{r-1}+1,F_r+1,F_{r+1}+1,F_{r+2}+1)\\
&\oplus\diag(F_{r-1}+2,F_r+2,F_{r+1}+2,F_{r+2}+2),\end{align*}
where for the direct sum $A\oplus B$ of two square matrices $A$ and $B$ we mean the block diagonal matrix having $A$ and $B$ on the diagonal. Note the congruence amongst the entries in $t$, $F_r=F_{r+2} \mod p$. Thus on each line, the last and pen-penultimate entries are the same modulo $p$. Furthermore, since this element does not centralise $\h$, it can have no lift to a diagonal cocharacter.

By assumption, $N_G(\h)=N_B(\h)$ is smooth. Thus $\la s\ra$ lifts to the image of a cocharacter $\chi'$ which, by Lemma \ref{NBhSmoothLift} is conjugate by $C_U(s)$ to a diagonal cocharacer $\chi$. Since by inspection, only five entries of $s$ are the same, $s$ is a regular toral element of $\n_\g(\h)$ and one checks \[C_U(s)=\la 1+te_{r,r+2}, 1+te_{2r+1,2r+3},1+te_{2r+6,2r+8},1+te_{2r+10,2r+12}:t\in k\ra.\]

The action of the second listed element in turn normalises $\h$ and the first, third and fourth simply change the values of $\lambda_1,\lambda_2,\lambda_3$. Thus if $g\in C_U(s)$ then one computes a new relation matrix $R'$ computing the normaliser $\n_\t(\h^g)$ which, by virtue of being independent of the values of $\lambda_i$, is identical to $R$. In particular, $\la s\ra$ still normalises $\h$ but there is still no lift to a diagonal cocharacter. This contradicts the conclusion of Lemma \ref{NBhSmoothLift}, hence $N_G(\h)$ is not smooth.
\end{example}

The next example will show the necessity of the bound in Theorem \ref{thmreductivenorms}. We first collect some miscellaneous auxiliary results in the following lemma. Recall that a subgroup $H$ of a connected reductive group $G$ is called  \emph{$G$-irreducible} if it is in no proper parabolic subgroup of $G$. 

\begin{lemma}\label{Girred}
Suppose $G$ is a connected reductive algebraic group and $H$ is a (possibly disconnected) closed reductive subgroup of $G$.

(i) We have $N_G(H)_\red^\circ=H^\circ C_G(H)_\red^\circ$. 

(ii) If $H$ is $G$-irreducible, then $C_G(H)_\red^\circ = \Rad(G)$, where $\Rad(G) = Z(G)^\circ_\red$.

(iii) Suppose $H \leq M \leq G$ is an intermediate reductive subgroup with $\Rad(G) \leq \Rad(M)$ and that $H$ is $G$-irreducible.
Write $Z(M)^\circ = \Rad(M) \times \mu_M$ for an infinitesimal subgroup scheme $\mu_M$.
Then either $\mu_M \leq Z(H)$ or $N_G(H)$ is non-smooth.
\end{lemma}
\begin{proof} (i) and (ii) follow from \cite[Lemmas 6.2 and 6.8]{Mar03}. 

For (iii), clearly $\mu_M \leq Z(M)\leq N_G(H)$.
If $N_G(H)$ is smooth, then by parts (i) and (ii) we have
$\mu_M \leq Z(M)^\circ \leq H^\circ C_G(H)_\red^\circ = H^\circ \Rad(G)$. 
This forces $\mu_M \leq H^\circ$.
\end{proof}

\begin{examples}\label{badchar} Lemma \ref{Girred} can be used to produce reductive subgroups $H$ of $G$ with non-smooth normalisers in bad characteristic. We use \cite[Example 4.1]{Her13}, in which the first author constructs examples of non-smooth centralisers for each
reductive group over a field of characteristic $p$ for which $p$ is not a very good prime for $G$.
All the subgroups constructed in \emph{loc.~cit.} are maximal rank reductive subgroups $M$ such that $C_G(M)=Z(M)$ is non-smooth,
hence $\mu_M \neq 1$ in Lemma \ref{Girred}(iii) above.
In many cases, we may take a further connected, reductive $G$-irreducible subgroup
$H$ of $M$ such that $p$ is pretty good for $H$. Thus its centre is in fact smooth, and being finite, cannot contain $\mu_M$.
Thus by Lemma \ref{Girred}(iii) the normaliser $N_G(H)$ is non-smooth.
Let us list some triples $(G,p,M,H)$ which work for this process. By $V_n$ we denote a natural module of dimension $n$ for the classical group $M$; by $\tilde M_1$ we mean a subgroup of type $M_1$ corresponding to short roots.

\begin{center}\begin{tabular}{c c c c}\hline
$G$ & $p$ & $M$ & $H$\\\hline
$G_2$ & $3$ & $A_2$ & $A_1\hookrightarrow M;\ V_3|_H=L(2)$\\
$F_4$ & $2$ & $A_1^4$ & $A_1\hookrightarrow M;\ x\mapsto (x,x^2,x^4,x^{16})$\\
$F_4$ & $3$ & $A_2\tilde A_2$ & $(A_1,A_1)\hookrightarrow M; (V_3,V_3)|_H=(L(2),L(2))$\\
$E_8$ & $5$ & $A_4^2$ & $A_1^2\hookrightarrow M;\ (V_5,V_5)|_H=(L(4),L(4))$\\
$\SL_p$ & $p>2$ & $\SL_p$ & $A_1\hookrightarrow M;\ V_p|_H=L(p-1)$.
\end{tabular}\end{center}
\end{examples}

\begin{remark}A complete list of conjugacy classes of simple $G$-irreducible subgroups of exceptional groups has been compiled by A.~Thomas, see \cite{Tho15} for the cases of rank at least $2$ and \cite{Tho16} for the rank $1$ case. For the $G_2$ example one may consult \cite[Theorem 1, Corollary 3]{SteG2}.\end{remark}

The next example shows the promised tightness of Theorem B(i) as stated in Remark \ref{thmbrems}(a).

\begin{lemma}\label{w1actingirred}Let $G=\GL(V)$ with $\dim V\geq p-1\geq 3$ and take any subspace $W\leq V$ with $\dim W=p-1$. Then if $W_1\leq\gl(W)$ is the first Witt algebra in its $p-1$-dimensional representation we have $N_G(W_1)$ is not smooth.\end{lemma}
\begin{proof}Since $W_1$ is irreducible on $W$, the normaliser $\n_{\gl(V)}(W_1)=\n_{\sl(W)}(W_1)\oplus \z\oplus \gl(U)$ for $V=W\oplus U$ and $\z$ the centre of $\gl(W)$. Moreover as $W_1$ is irreducible on $W$, so is $\n=\n_{\sl(W)}(W_1)$. By Theorem C, $\n$ is semisimple, hence, as  $W_1$ is simple, it must be a direct factor of $\n$, say $\n=W_1\oplus \h$. But now the action of $\ad\h$ on $W$ is a $W_1$-module homomorphism, hence is a scalar by Schur's lemma. Thus $\h\leq\z(\sl(W))=0$. It follows that $\n=W_1$.

Now $N_G(W_1)$ sends $W$ to another $W_1$-invariant subspace of the same dimension, hence $N_G(W_1)\leq \GL(W)\times\GL(U)$. Since $W_1$ is self-normalising, if $N_G(W_1)$ were smooth we would have $\Lie N_G(W_1)=\n_\g(W_1)=W_1\oplus \gl(U)$. This shows that $W_1$ is algebraic, a contradiction.\end{proof}

We now justify the remark after Theorem B that the bound in Theorem B(i) is tight for $G=\Sp_{2n}$.

\begin{lemma}\label{w1insp}The $p$-dimensional Witt algebra $W_1$ is a maximal subalgebra of $\sp_{p-1}$. Furthermore, its normaliser in any $\Sp_{p-1}$-Levi of $\Sp_{2n}$ with $2n\geq p-1$ is non-smooth.\end{lemma}
\begin{proof}Since $W_1$ stabilises the element \[X\wedge X^{p-1}+\frac{1}{2}X^2\wedge X^{p-2}+\frac{1}{3}X^3\wedge X^{p-3}+\dots+\frac{2}{p-1}X^{(p-1)/2}\wedge X^{(p+1)/2}\in {\bigwedge}^2 V\] we find that $W_1$ is contained in $\sp_{p-1}$, acting irreducibly on the $p-1$-dimensional module. Exponentiating 
a set of nilpotent generators of the Witt algebra as in the proof of
Theorem B(ii) gives an irreducible subgroup $W\leq \Sp_{p-1}$.
We claim that we must have equality.
From this claim it follows that $W_1$ is in no proper classical algebraic subalgebra of $\sp_{p-1}$, hence, by Theorem \ref{Strade}, is maximal.

To prove the claim, suppose $W$ is a proper subgroup of $G=\Sp_{p-1}$. Since $W$ is irreducible on the $p-1$-dimensional module, $W$ is it no parabolic of $G$. Thus it is in a connected reductive maximal subgroup $M$. We must have that $M$ is simple, or else $W_1$ would be in a parabolic of $G$. Now since the lowest dimensional non-trivial representation of $W_1$ is $p-1$, it follows that $M$ can have no lower-dimensional non-trivial representation. Since $p>2$, $\Sp_{p-1}$ has no simple maximal rank subgroup. All classical groups of rank lower than $\frac{p-1}{2}$ have natural modules of smaller dimension than $p-1$, so $M$ is of exceptional type. The lowest dimensional representations of the exceptional types are $6$ ($p=2$), $7$, $25$ ($p=3)$, $26$, $27$, $56$ and $248$. The only time one of these is $p-1$ is when $p=57$ and $M=E_7$. But if $p=57$ then $p>2h-2$ for $E_7$, then by Theorem B(ii) all maximal semisimple subalgebras are algebraic and so $W_1$ is not a subalgebra of $E_7$. This proves the claim, hence gives the first part of the lemma.

For the second, with $2n>p-1$, we have $W_1\leq \sp_{p-1}\oplus\sp_{2n-p+1}$ with $W_1$ sitting in the first factor. Then its normaliser is evidently $W_1\oplus \sp_{2n-p+1}$, however this is not algebraic for $p>3$, hence the normaliser $N_G(W_1)$ cannot be smooth. Thus we have shown that normalisers of all subalgebras of $\sp_{2n}$ are smooth only if $p>h+1$.
\end{proof}

If $n\geq p$ there is a more straightforward example of a (non-restricted) subalgebra of $\gl_n$ whose normaliser in $\GL_n$ is not smooth.

\begin{example} Let $\g=\gl_n$, take $J_p$ a Jordan block of size $p$ and take the abelian one-dimensional Lie algebra $\h=k(I_p+J_p)$ where $I_p$ is an identity block of size $p$. Then one can show with elementary matrix calculations that the normaliser of $N_G(\h)$ is non-smooth.\end{example}

The next example shows that even the normalisers of smooth groups are not smooth, even in $\GL(V)$, and even when $p$ is arbitrarily large.

\begin{lemma}\label{example:smoothUnipotent}Let $G=\GL(V)$ with $\dim V\geq 3$ and let $W$ be a $3$-dimensional subspace. Let $U\leq \GL(W)$ be defined as the smooth subgroup whose $k$-points are \[U(k)=\left\{\left[\begin{array}{c c c}1 & 0 & t\\0 & 1 & t^p\\ 0 & 0 & 1\end{array}\right]:t\in k\right\}.\]Write $V=W\oplus W'$ for some complement $W'$ to $W$ and set $H=U\oplus \GL(W')\leq \GL(V)$. Then $N_G(H)$ is non-smooth.\end{lemma}
%\begin{proof} From the reductivity of $\GL(W')$ it follows that $N_G(H)=N_{\GL(W)}(U)\oplus \GL(W')$ so it suffices to show that $N_{\GL(W)}(U)$ is non-smooth. This is a routine calculation. For example, if $x$ is an element of a $k$-algebra $A$, then one checks \[\left[\begin{array}{c c c}1 & x & 0\\0 & 1 & 0\\ 0 & 0 & 1\end{array}\right]\in N_{\GL(W)}(U)\] provided $x^p=0$. This is not a square-free polynomial, hence defines a non-smooth group-scheme. In particular, the matrix $\left[\begin{smallmatrix}0 & 1 & 0\\0 & 0 & 0\\0 & 0 & 0\end{smallmatrix}\right]$ is an element of $\Lie N_{\GL(W)}(U)$, but not of $\Lie N_{\GL(W)}(U)(k)$.  \end{proof}
\begin{proof} From the reductivity of $\GL(W')$ it follows that $N_G(H)=N_{\GL(W)}(U)\oplus \GL(W')$ so it suffices to show that $N_{\GL(W)}(U)$ is non-smooth. This is a routine calculation. For example, if $x$ is an element of a $k$-algebra $A$, with $x^p=0$ then one checks that the matrix \[\left[\begin{array}{c c c}1 & x & 0\\0 & 1 & 0\\ 0 & 0 & 1\end{array}\right]\in N_{\GL(W)}(U)(A).\] Now, the normaliser of $U$ of course normalises $\Lie(U)$. Since \[\Lie(U)=k\left[\begin{array}{c c c}0 & 0 & 1\\0 & 0 & 0\\ 0 & 0 & 0\end{array}\right],\] the normaliser of $\Lie(U)$ is the product of the centraliser of a certain (nilpotent) element and the image of a cocharacter associated with that element. In particular, the normaliser of $\Lie(U)$ is contained in the upper triangular Borel subgroup.

Write $V$ for the unipotent radical of that Borel subgroup, so $V$ is $3$-dimensional; a typical element has the form \[\left[\begin{array}{c c c}1 & a & b\\0 & 1 & c\\ 0 & 0 & 1\end{array}\right].\] In fact, the condition $a^p = 0$ defines the scheme-theoretic normaliser in $V$ of $U$, and the condition $a = 0$ defines the corresponding smooth subgoup of $V$ whose $k$-points form the group-theoretic normaliser of $U(k)$ in $V(k)$. The lemma follows.\end{proof}

Now we show that normalisers of height two or more subgroup schemes are not smooth.

\begin{example} Let $G$ be any connected reductive algebraic group over an algebraically closed field $k$ of characteristic $p>2$ and set $F:G\to G$ to be the Frobenius endomorphism. Let $B=TU$ be a Borel subgroup of $G$ with $T$ an $F$-stable maximal torus, and let $U$ the non-trivial unipotent radical. Let $T_r$ be the kernel in $T$ of $F^r$ and $U_1$ the kernel in $U$ of $F$. Finally set  $H=T_r\ltimes U_1$. Then $N_G(H)=T\ltimes U_1$, hence is not smooth.\end{example}

The next example shows that if $p=\dim V$, then the normaliser of a smooth connected solvable non-diagonalisable
algebraic subgroup of $\GL(V)$ can even be irreducible on $V$,
thus \emph{a fortiori} it is not smooth.
This also gives an example for when $p=2$ and $\dim V=2$ that the normalisers in $\SL(V)$ and $\GL(V)$ of subalgebras of the respective Lie algebras are not smooth.

\begin{example}\label{w1pluso1example} By \cite[Lemma 3]{Ten87} the Lie algebra $W_1+O_1$ formed as the semidirect product of $W_1$ and $O_1$, where $O_1$ acts on itself by multiplication, is a maximal subalgebra of $\sl_p=\sl(k[X]/X^p)$.
We imitate the embedding of $O_1$ in $\gl_p$ by a solvable subgroup of $\GL_p$.
Define the height $\mathrm{ht}(\alpha)$ of a root $\alpha$ to be the sum of the coefficients of the simple roots. Let $U$ be the subgroup $\langle \prod_{\alpha\in R^-; \mathrm{ht}(\alpha)=i}x_{\alpha}\rangle_{1\leq i\leq p-1}$. By construction $U$ is connected and unipotent and one can show that $\dim U=p-1$ and that $\Lie H=O_1$,
where $H$ is the smooth solvable subgroup
$Z(\GL_p)U$. Now it can be shown that there is a subgroup scheme $W$ corresponding to $W_1$ in $\GL_p$ which normalises $H$ and for which $W\ltimes H$ is irreducible. It immediately follows that $N_G(H)$ cannot be smooth.\end{example}

Finally we show that if $p\leq 2n-1$ the normalisers in $\GL_n$ and $\SL_n$ of subspaces of their Lie algebras are not all smooth, even when these normalisers are generated by nilpotent elements, showing that the bound in Theorem B(ii) cannot be improved for general subspaces.

\begin{lemma} If $p<2n-1$, normalisers of subspaces of $\gl_n$ (or $\sl_n$) are not necessarily smooth.\end{lemma}
\begin{proof}Let $p=2n-3$ and let $\h=\sl_2=\Lie H$ with $H=\SL_2$ over a field $k$ of characteristic $p$. Then the action of $H$ on the simple module $L((p+1)/2)$ gives an (irreducible) embedding $H\to\GL_n$. Restricting the adjoint representation of $\gl_n$ on itself to $H$ gives a module \[L((p+1)/2)\otimes L((p+1)/2)^*\cong T(p+1)\oplus M,\] where $M$ is a direct sum of irreducibles for $H$ (and $\h$) and $T(p+1)$ is a tilting module, uniserial with successive composition factors $L(p-3)|L(p+1)|L(p-3)$. 

Now for the algebraic group $H=SL_2$ we have $L(p+1)\cong L(1)\otimes L(1)^{[1]}$ by Steinberg's tensor product formula. Restricting to $\h$, $L(p+1)$ is isomorphic to $L(1)\oplus L(1)$. Now it is easy to show the restriction map $\Ext_G^1(L(p+1),L(p-3))\to\Ext_\g^1(L(1),L(p-3))\oplus \Ext_\g^1(L(1),L(p-3))$ is injective. Hence $T(p+1)|_\g$ contains a submodule $M$ isomorphic to $L(1)/L(p-3)$.

Now, the Lie theoretic normaliser of $M$ contains $\h$ but the scheme-theoretic stabiliser does not contain $H$. It follows that the normaliser of this subspace is not smooth.

Indeed, as $\h$ acts irreducibly on the $n$-dimensional natural representation for $\gl_n$, it is in no parabolic of $\gl_n$ (or $\sl_n$). However, the set of $k$-points $N_H(M)(k)=N_{\GL_n}(M)(k)\cap H$ is in a parabolic of $H$, hence in a parabolic of $\GL_n$.\end{proof}

{\footnotesize
\bibliographystyle{amsalpha}
%%\bibliography{bib}}
\providecommand{\bysame}{\leavevmode\hbox to3em{\hrulefill}\thinspace}
\providecommand{\MR}{\relax\ifhmode\unskip\space\fi MR }
\providecommand{\MRhref}[2]{%
  \href{http://www.ams.org/mathscinet-getitem?mr=#1}{#2}
}
\providecommand{\href}[2]{#2}

\end{document}